\documentclass{lms}
\usepackage{amsfonts}
\usepackage{amsmath}
\usepackage{amssymb}
\usepackage{graphicx}
\newtheorem{theorem}{Theorem}[section]
\newtheorem{proposition}[theorem]{Proposition}
\newtheorem{lemma}[theorem]{Lemma} 
\newtheorem{corollary}[theorem]{Corollary}

\newnumbered{remark}[theorem]{Remark}
\newnumbered{question}[theorem]{Question}
\newnumbered{examples}[theorem]{Examples}
\newnumbered{definition}[theorem]{Definition}

\newcommand{\CAT}{{\mathrm{CAT}}}
\newcommand{\wreath}{{\mathrm{wr}}}
\newcommand{\bbar}{\underline{{\mathrm B}}}
\newcommand{\ebar}{\underline{{\mathrm E}}}
\newcommand{\bfc}{{\mathbf{C}}}
\newcommand{\cc}{\mathcal{C}}
\newcommand{\cs}{\mathcal{S}}
\newcommand{\ind}{{\mathrm{ind}}}
\newcommand{\Lk}{{\mathrm{Lk}}}
\newcommand{\link}{\Lk}
\newcommand{\wtX}{{\widetilde X}}
\newcommand{\inverselim}{\mathop{\mathrm{inv.lim.}}}
\newcommand{\cala}{{\mathcal A}}
\newcommand{\calc}{{\mathcal C}}
\newcommand{\cald}{{\mathcal D}}
\newcommand{\calf}{{\mathcal F}}
\newcommand{\cali}{{\mathcal I}}
\newcommand{\calk}{{\mathcal K}}
\newcommand{\calt}{{\mathcal T}}
\newcommand{\frakf}{{\mathfrak F}}
\newcommand{\hatt}{{\widehat t}}
\newcommand{\hatT}{{\widehat T}}
\newcommand{\nstar}{{\mathrm{St}}}
\newcommand{\pt}{{*}}
\newcommand{\tilg}{{\widetilde G}}
\newcommand{\sing}{{\mathrm{Sing}}}
\newcommand{\ff}{{\mathbb F}}
\newcommand{\nn}{{\mathbb N}}
\newcommand{\qq}{{\mathbb Q}}
\newcommand{\rr}{{\mathbb R}}
\newcommand{\zz}{{\mathbb Z}}
\newcommand{\va}{{\mathbf a}}
\newcommand{\vb}{{\mathbf b}}
\newcommand{\vg}{{\mathbf g}}
\newcommand{\vq}{{\mathbf q}}
\newcommand{\red}{{\mathit r}}

\newcommand\mapright[1]{\smash{\mathop{\longrightarrow}\limits^{#1}}}
\newcommand\mapdown[1]{\Big\downarrow\rlap{$\vcenter{\hbox{$\scriptstyle#1$}}$}}

\title{A metric Kan-Thurston theorem}

\classno{20F67 (primary), 20J06, 55P99, 57S30 (secondary)}  
\author{Ian J. Leary}

\extraline{Partially supported by NSF grant DMS-0804226 and by
      the Heilbronn Institute.  Some of this work was done at MSRI,
      Berkeley, where research is supported by NSF grant DMS-0441170.}

\date{\today}

\begin{document} 

\maketitle



\begin{abstract} 
For every simplicial complex $X$, we construct a locally CAT(0)
cubical complex $T_X$, a cellular isometric involution $\tau$ on 
$T_X$ and a map $t_X:T_X\rightarrow X$ with the following properties: 
$t_X\tau = t_X$; $t_X$ is a homology isomorphism; the induced map from
the quotient space $T_X/\langle\tau\rangle$ to $X$ is a homotopy
equivalence; the induced map from the fixed point space $T_X^\tau$ to
$X$ is a homology isomorphism.  The construction is functorial in
$X$. 

One corollary is an equivariant Kan-Thurston theorem:
every connected proper $G$-CW-complex has the same equivariant
homology as the classifying space for proper actions, $\ebar \tilg$,
of some other group $\tilg$.  From this we obtain extensions of a 
theorem of Quillen on the spectrum of a (Borel) equivariant cohomology 
ring and a result of Block concerning assembly conjectures.  
Another corollary of our main result is that there can be
no algorithm to decide whether a CAT(0) cubical group is generated by
torsion.

In appendices we prove some foundational results concerning cubical
complexes, including the infinite dimensional case.  We characterize
the cubical complexes for which the natural metric is complete; we
establish Gromov's criterion for infinite-dimensional cubical
complexes; we show that every CAT(0) cube complex is cubical; we
deduce that the second cubical subdivision of any locally CAT(0) cube
complex is cubical.
\end{abstract}


\section{Introduction} 

In \cite{KT}, D. Kan and W. Thurston showed that every simplicial
complex $X$ has the same homology as some aspherical space $T_X$, or
equivalently every connected simplicial complex $X$ has the same
homology as the classifying space for some discrete group $G_X$.  
More precisely, they give a map $t_X:T_X\rightarrow X$ which induces a
homology isomorphism for any local coefficient system on $X$.  
Furthermore, the construction of $T_X$ and $t_X$ has good naturality
properties.  The Kan-Thurston theorem has been extended by a number of
authors~\cite{BDH,Hausmann,Maunder}.  In particular, Baumslag Dyer and
Heller showed that in the case when $X$ is finite, the space $T_X$ may 
be taken to be a finite simplicial complex~\cite{BDH}.  (In contrast, 
the original $T_X$ was uncountable whenever $X$ had dimension at least
two.)  

The Kan-Thurston theorem has also been generalized in a number of
ways.  McDuff showed that every $X$ has the homotopy type of the 
classifying space of some discrete monoid
$M_X$~\cite{McDuff}.  Leary and Nucinkis showed
that every $X$ has the homotopy type of a space of the form 
$\ebar G/G$ for some discrete group $G$, where $\ebar G$ denotes the 
classifying space for proper actions of $G$~\cite{LearyNucinkis}.  

Our purpose is to give new proofs of some of these results using
methods that come from metric geometry.  We give an extremely 
short summary of these methods before stating our result.  A
good general reference for this material is~\cite{brihae}.

A geodesic metric space is a metric space in which any two points are
joined by a `geodesic', i.e., a path of length equal to their
distance.  A geodesic metric space is CAT(0) if any geodesic triangle
is at least as thin as a comparison triangle in the Euclidean plane
having the same side lengths.  Any CAT(0) space $X$ is contractible: 
for any point $x\in X$ one may define a contraction of $X$ to $x$ by 
moving points at constant speed along the (unique) geodesic path
joining them to $x$.  

Let $T$ be a cubical complex.  The link of any vertex in $T$ is 
naturally a simplicial complex, and vertex links in the 
universal cover $\widetilde T$ are isomorphic to vertex links in $T$.  
There is a metric on $T$ in which distances are defined in
terms of the lengths of piecewise-linear paths, using an isomorphism 
between each cube of $T$ and a standard Euclidean cube of side length 
one to measure the length of paths.  Gromov has given a simple
criterion for this metric to be CAT(0): this happens if and only if 
$T$ is simply connected and the link of each vertex is a flag
complex~\cite{brihae,gromov}.  In particular, it follows that if all 
vertex links in $T$ are flag complexes, then $T$ is aspherical.  

The standard references for Gromov's criterion consider only the 
case when $T$ is finite-dimensional (see
\cite[theorem~II.5.20]{brihae}, \cite[section 4.2.C]{gromov}, 
\cite[ch.~2]{bridson} or the unpublished \cite{moussong} for example).  
In Theorem~\ref{gromovthm} in an appendix to this 
paper we give a proof in the general case.  

Let $T$ be a locally CAT(0) cube complex.  There are many 
known constraints on the fundamental group $G$ of such a $T$.  
For example, if $T$ is finite, then $G$ is 
bi-automatic, which implies that there are good algorithms 
to solve the word and conjugacy problems in $G$~\cite{nibloreeves}.  
As another example, when $T$ is countable, the group $G$ has 
the Haagerup property, i.e., $G$ admits a proper affine isometric 
action on a Hilbert space.  This implies that $G$ can contain 
no infinite subgroup having Kazhdan's property (T)~\cite{nibloreevesII}.  

In contrast, we show that there is no restriction whatsoever 
on the homology of a locally CAT(0) cubical complex.  We also 
show that there is no restriction on the homotopy type of the 
quotient of a locally CAT(0) cubical complex by an isometric 
cellular involution, apart from the obvious condition that 
it should be homotopy equivalent to a CW-complex.  
The following statement is a version of our main theorem.

\begin{theorem*}[A]
Let $X$ be a simplicial complex.  There is a locally CAT(0) 
cubical complex $T_X$ and a map $t_X:T_X\rightarrow X$ with 
the following properties. 
\begin{enumerate} 

\item \label{conda}
The map $t_X$ induces an isomorphism on homology for any 
local coefficients on $X$.  

\item \label{condb}
There is an isometric cellular involution $\tau$ on $T_X$ 
so that $t_X\circ \tau= t_X$ and the induced map 
$T_X/\langle \tau\rangle \rightarrow X$ is a homotopy equivalence. 

\item 
The map $t_X:T_X^\tau
\rightarrow X$ induces an isomorphism on homology for any local
coefficients on $X$, where $T_X^\tau$ denotes the fixed point set 
in $T_X$ for the action of $\tau$.  

\end{enumerate} 
If $X$ is finite, then so is $T_X$.  The construction is functorial 
in the following senses: 
\begin{itemize}

\item Any simplicial map $f:X\rightarrow Y$ that is injective on each 
simplex gives rise to a $\tau$-equivariant cubical map
$T_f:T_X\rightarrow T_Y$.   

\item In the case when $f$ is injective, $T_f$ embeds $T_X$
  isometrically as a totally geodesic subcomplex of $T_Y$.  

\item In the case when $f$ is locally injective, $T_f$ is a locally
  isometric map.  In particular, in this case $T_f$ is injective 
  on fundamental groups.

\end{itemize}  

If $X$ is the union of subcomplexes $X_\alpha$,
then $T_X$ is equal to the union of the $T_{X_\alpha}$.  
The dimension of $T_X$ is equal to the dimension of $X$, except that 
when $X$ is 2-dimensional, $T_X$ is 3-dimensional.  In all cases, 
the dimension of $T_X^\tau$ is equal to the dimension of $X$.  
\end{theorem*}

By a locally injective map of simplicial complexes, we mean a map 
$f:X\rightarrow Y$ which is injective on each simplex of $X$, and 
for each vertex $x\in X$ induces an injective map from the link of 
$x$ to the link of $f(x)$.  Equivalently, a locally injective map 
of simplicial complexes is a simplicial map which induces a locally
injective map of realizations.


The existence of $T_X$ and $t_X$ having property~\ref{conda} is a
strengthening of the Baumslag-Dyer-Heller version of the Kan-Thurston
theorem~\cite{BDH,KT}, since the only groups used are fundamental 
groups of locally CAT(0) cubical complexes.  Similarly, 
the existence of $T_X$ and $t_X$ having
property~\ref{condb} is a strengthening of the theorem of Leary and
Nucinkis on possible homotopy types of $\ebar
G/G$ since the only groups used are groups that act 
with stabilizers of order one and two on CAT(0) cubical
complexes~\cite{LearyNucinkis}.  This follows from the fact that 
whenever a group $G$ acts with finite stabilizers on a CAT(0) 
cubical complex $X$ that cubical complex is a model for $\ebar G$
(see Theorem~\ref{ebarg}).  

We describe a number of corollaries to Theorem~A in Sections 
\ref{sec:nin}--\ref{sec:twe} of this paper.  We believe 
that none of these results follow from earlier versions of the 
Kan-Thurston theorem, and we hope that they go some way towards 
motivating our work.  

One advantage of Theorem~A over other versions of the Kan-Thurston
theorem is that it readily yields an equivariant version, which 
we state as Theorem~\ref{thm:genequiv}.  If $G$ acts on $X$ which
we suppose to be connected, then $G$ acts also on $T_X$.  Let $E$ be the 
universal covering space of $T_X$, and let $\tilg$ be the group 
of self-homeomorphisms of $E$ that lift the action of $G$ on $T_X$.  
If we assume that $G$ acts on $X$ with finite stabilizers, then 
$\tilg$ acts with finite stabilizers on the CAT(0) cubical complex
$E$.  It follows (see Theorem~\ref{ebarg}) that $E$ is a model for 
$\ebar \tilg$, the classifying space for proper actions of $\tilg$. 
By construction, there is a surjective group homomorphism 
$\tilg\rightarrow G$, and an equivariant map $\ebar\tilg\cong 
E\rightarrow X$ which is an equivariant homology isomorphism.  

Our construction of $t_X$ and $T_X$ having property~1 is closely
modelled on the construction of groups used by Maunder
in~\cite{Maunder} and by Baumslag, Dyer and Heller in~\cite{BDH}.  
The main ingredients in the proof in~\cite{BDH} can be illustrated by
solving a simpler problem: suppose that one wants to construct, for 
each $n>0$, a group $G_n$ so that the integral homology of $G_n$ is 
isomorphic to the homology of the $n$-sphere.  To solve this problem, 
they fix an acyclic group $H$, together with an element $h\in H$ of 
infinite order.  Take $G_1$ to be the subgroup of $H$ generated 
by $h$, and define groups $H_n$ (to play the role of the $n$-disc)
and $G_n$ for $n\geq 2$ as follows.  Let $H_2=H$, our fixed acyclic 
group.  Now define $G_2$ by gluing two 2-discs along a circle: 
$G_2=H*_{G_1}H$.  Define $H_3=H_2*_{G_1}(G_1\times H)$, and check 
that $H_3$ is an acyclic group containing $G_2$ as a subgroup.  
(The standard copy of $H_2$ and its conjugate by $(1,h)\in
G_1\times H$ together generate a subgroup of $H_3$ isomorphic to $G_2$.)  
The definitions of the higher groups are now clear: 
$G_{n+1}=H_{n+1}*_{G_n}H_{n+1}$ and $H_{n+1}=
H_n*_{G_{n-1}}(G_{n-1}\times H)$.  

This construction carries over without change to the metric world, 
once one has found an acyclic locally CAT(0) cubical complex $A$ to 
replace the classifying space for $H$.  The proof from~\cite{Maunder} 
when translated into the metric world gives a construction of the 
space denoted by $T_X^\tau$ in the statement of
Theorem~A.  To obtain $T_X$ and the involution
$\tau$, we need a second acyclic locally CAT(0) cubical complex 
$A'$, containing $A$ and equipped with an involution $\tau$ such 
that $A$ is equal to the fixed point set $A'^\tau$ and such that the 
quotient space $A'/\langle \tau\rangle$ is contractible.  Roughly 
speaking, $T_X$ is constructed by attaching a copy of $A'$ to every
copy of $A$ that appears in the construction of $T_X^\tau$.  
The acyclic locally CAT(0) 2-complex $A$ that we construct is a 
presentation 2-complex, in which the 2-cells are large, 
right-angled polygons which in turn are made from smaller squares.

The remainder of this paper is structured as follows.  
In Section~\ref{sec:polygons}, we consider the problem of making 
a CAT(0) polygon with given side lengths from unit squares.  In 
Section~\ref{sec:acyc} we build acyclic locally CAT(0) 2-complexes 
that will be used in the construction of $A$ and $A'$.  
Section~\ref{sec:conacyc} constructs 
$A$, $A'$ and $\tau$ as described in the previous paragraph, 
and Section~\ref{sec:main} gives the proof of Theorem~A, 
following~\cite{Maunder}.  

Sections \ref{sec:sev}~and~\ref{sec:eig} consider variants on the 
construction of $T_X$, including a cube complex whose cubical
subdivision is $T_X$, a version $T'_X$ which is always metrically 
complete but for which the map $T'_X\rightarrow X$ is not necessarily
proper, and a version $\hatT_X$ defined for arbitrary spaces $X$, 
and such that the map from $\hatT_X/\langle \tau\rangle$ to $X$ is 
a weak homotopy equivalence.  Section~\ref{sec:nin} contains 
Theorem~\ref{thm:genequiv}, an equivariant Kan-Thurston theorem 
for actions with finite stabilizers.  Sections
\ref{sec:ten}~and~\ref{sec:ele} discuss applications of
Theorem~\ref{thm:genequiv} to Borel equivariant cohomology and to
assembly conjectures.  In particular, Corollary~\ref{cor:twentytwo} 
is an extension to Quillen's theorem on the spectrum of a Borel
equivariant cohomology ring~\cite{quillen}, and Theorem~\ref{thm:assembly} 
generalizes an observation of Block~\cite[Introduction]{block}.  
Block's observation concerned the Baum-Connes conjecture for
torsion-free groups, while our Theorem~\ref{thm:assembly} applies to 
groups with torsion and to many assembly conjectures.  
In Section~\ref{sec:twe} we use Theorem~A to deduce that there 
can be no algorithm to decide whether a CAT(0) cubical group 
is generated by torsion.  Section~\ref{sec:thi} discusses some 
open questions related to our work.  

Sections \ref{sec:appone},
\ref{sec:apptwo}~and~\ref{sec:appthree} contain results about cubical 
complexes with an emphasis on the infinite-dimensional case, 
and should be viewed as appendices to the paper.  In some cases
the results for finite-dimensional cubical complexes are well-known, 
and our proofs proceed by reducing the arbitrary case to the
finite-dimensional case.  In Theorem~\ref{thm:cubecomplete} we 
characterize the cubical (and simplicial) complexes for which 
the natural path metric is complete.  In Theorem~\ref{gromovthm} 
we establish Gromov's criterion: a simply-connected cubical 
complex is CAT(0) if and only if all links are flag complexes.  
In Theorem~\ref{ebarg} we show that whenever a finite group 
acts on a CAT(0) cubical complex (not assumed to be metrically 
complete), the fixed point set is non-empty.  
Section~\ref{sec:appthree} concerns the comparison between cubical 
complexes and the more general complexes often called cube complexes.  
The result is that in the CAT(0) case there is almost no advantage 
to using cube complexes rather than cubical complexes.  
In Theorem~\ref{thm:cubevcubical}, we show that every CAT(0) cube 
complex is in fact cubical.  In Corollary~\ref{cor:secsubdiv} we 
deduce that the second cubical subdivision of every locally CAT(0) 
cube complex is cubical.  

The results in Sections~\ref{sec:appone},
\ref{sec:apptwo}~and~\ref{sec:appthree}, will not be surprising to
experts, but we know of no published proof the main results that 
covers the infinite-dimensional case.  In a private communication in 2007,
Michah Sageev told the author that Yael Algom-Kfir had obtained a
proof of Gromov's criterion for infinite-dimensional cube complexes
(proved here as Theorem~\ref{gromovthm}) while working on her MSc with
him.  Also in 2007, Yael Algom-Kfir told the author that she did not
intend to publish her proof.  Our proof of Theorem~\ref{gromovthm}
involves a reduction to the case of a finite-dimensional cubical
complex, and we rely heavily on results of Sageev~\cite{sag} in making
this reduction.

\begin{acknowledgements}
This work was started in 2001, when the author visited Wolfgang L\"uck 
in M\"unster and Jean-Claude Hausmann in Geneva.  Some progress was 
made at MSRI in 2007, and the author thanks these people and
institutes for their hospitality.  The author also thanks many 
other people for helpful conversations concerning this work, 
especially Graham Niblo, Boris Okun and Michah Sageev.  Finally, 
the author thanks Raeyong Kim and Chris Hruska for their helpful
comments on earlier version of this paper.  
\end{acknowledgements}

\section{Tesselating CAT(0) polygons} \label{sec:polygons}

We seek CAT(0) tesselations by squares of polygons with given side
lengths.  A tesselated $n$-gon is defined to be a 2-dimensional 
cubical complex $S$ which is homeomorphic to a disc, together with an 
ordered sequence of $n$ distinct vertices $v_1,\ldots,v_n$ on the 
boundary circle $\partial S$ of $S$.  It will be convenient to view 
the index set for these vertices as $\zz/n$ rather than $\zz$, so 
that $v_{n+k}=v_k$.  Call the vertices $v_i$ the {\sl corners} of 
the polygon, and call the complementary pieces of $\partial S$ 
the (open) {\sl sides} of the polygon.  The $i$th side 
is the piece whose closure contains $v_i$ and $v_{i+1}$.  
The $i$th {\sl side length}
$l_i$ is defined to be the number of edges in the $i$th side.  

Suppose that $(S,v_1,\ldots,v_n)$ is a tesselated polygon.  For $v$ a
vertex of $S$, let $n(v)$ denote the degree of $v$, i.e., the number 
of neighbouring vertices.  Define the curvature $c(v)$ of $S$ at $v$ 
to be 
\[
c(v) = \left\{\begin{array}{ll}
4-n(v)&\mbox{if $v\notin \partial S$}\\
3-n(v)&\mbox{if $v\in \partial S$ is not a corner of $S$} \\
2-n(v)&\mbox{if $v$ is a corner of $S$}
\end{array}\right.
\]
We say that $(S,v_1,\ldots,v_n)$ is a CAT(0) tesselated polygon 
if for every vertex $v\in S$, $c(v)\leq 0$.  By Gromov's criterion, 
it follows that the natural path-metric on $S$ is a CAT(0) metric 
in which each of the sides of $S$ is a geodesic.  

An easily proved combinatorial version of the Gauss-Bonnet theorem 
shows that for any $n$-gon $S$, 
\[ 
\sum_{v\in S} c(v) = 4-n.
\] 
It follows easily that the only tesselated CAT(0) quadrilaterals are 
rectangles with side lengths $(l,l',l,l')$, since a CAT(0)
quadrilateral $S$ must have $c(v)=0$ for all $v\in S$.  Similarly, a
tesselated CAT(0) pentagon can have $c(v)\neq 0$ for just one
vertex~$v$.

Clearly, some restrictions are required on the side lengths $l_i$ 
for the construction of a CAT(0) $n$-gon.  
For example, the perimeter $p$, defined by $p=\sum_i l_i$ 
must be even since every edge not in $\partial S$
belongs to exactly two squares.  Note also it cannot be the
case that any $l_i$ is greater than or equal to $p/2$, since each 
side of the polygon is a geodesic.  We do not characterize the 
sequences $(l_1,\ldots,l_n)$ for which $S$ exists, but just give 
some constructions that suffice for our purposes.  

First we give three ways to obtain a new CAT(0) polygon, starting 
from a CAT(0) $n$-gon with side lengths $(l_1,\ldots,l_n)$.  
 
One can introduce a new corner in the middle of the $i$th side 
of the $n$-gon, to produce an $(n+1)$-gon: if $k$ satisfies $0<k<l_i$ 
then an $(n+1)$-gon with side lengths $(l_1,\ldots,l_{i-1}, 
k,l_i-k,\ldots,l_n)$ is obtained.  

One can add a $1\times l_i$ rectangle to the $i$th side of 
the $n$-gon, producing a new $n$-gon with side lengths 
$(l_1,\ldots,l_{i-1}+1,l_i,l_{i+1}+1,\ldots,l_n)$.  By applying 
this process once for each $i$, one obtains an $n$-gon with 
side lengths $(l_1+2,\ldots,l_n+2)$ consisting of the original 
$n$-gon with a `collar' attached.  

One can take the cubical subdivision of an $n$-gon, producing an
$n$-gon with side lengths double those of the original $n$-gon.  

Next we discuss two ways to build a CAT(0) polygon, starting 
just from the required side lengths.  

Suppose that $R$ is a Euclidean rectangle tesselated by 
squares, with perimeter $p$ and corners $w_1,\ldots,w_4$.  Any
choice of $n$ vertices $v_1,\ldots,v_n$ in $\partial R$ which 
gives the required side lengths gives a way to view $R$ as a 
tesselated $n$-gon.  However, this $n$-gon will fail to be 
CAT(0) if $\{w_1,\ldots,w_4\}\not\subseteq \{v_1,\ldots,v_n\}$.  
Define subrectangles $R'_1,\ldots,R'_4$ of $R$ as follows: if 
there exists $i$ such that $w_j=v_i$, then let $R'_j$ be the 
single point $w_j$; if $w_j$ lies in the interior of the $i$th 
side of $R$, let $R'_j$ be the subrectangle of $R$ with 
$w_j$, $v_i$ and $v_{i+1}$ as three of its vertices.  
Provided that the four rectangles $R'_1,\ldots,R'_4$ are 
disjoint, we define $S$ to be the closure in $R$ of 
$R-\bigcup_{j=1}^4R'_j$.  This $S$ is a CAT(0) $n$-gon with 
the required side lengths.  Note that the curvature of $S$ 
is concentrated along $\partial S$, since each interior 
vertex of $S$ has valency 4.  

For example, consider the problem of realizing a pentagon with side 
lengths $(2,2,1,2,1)$.  If we start with a square of side 2, there 
are two distinct ways to cut up the boundary.  One of these does
not work, since two of the subrectangles $R'_j$ and $R'_{j+1}$
corresponding to adjacent vertices meet.  The other way gives a 
valid $S$ in the form of an `L'-shape consisting of three squares.  

As a second example, consider the problem of realizing a hexagon with 
sides $(2,1,1,2,1,1)$.  Again, there are two ways to break up the 
boundary of a square $R$ of side 2, one of which gives $S=R$, with 
four of the hexagon's vertices coincident with vertices of $R$ and 
two on the midpoints of opposite edges.  The other way does not 
work, since for some $j$ the subrectangles $R'_j$ and $R'_{j+2}$
meet, and so the construction gives a bow-tie consisting of two 
squares joined at a vertex.  The defect with this `hexagon' is that 
the midpoints of the two long sides coincide.  There is however another 
way to realize the hexagon with sides $(2,1,1,2,1,1)$, as a suitably 
labelled $1\times 3$ rectangle.  

As a third example, Figure~\ref{fig:one} illustrates a CAT(0) octagon
with seven sides of length 2 and one side of length 4 that was
constructed in this way.  Note that the curvature in this octagon is 
concentrated at two of its corners and at the midpoints of two of its 
sides.

\begin{figure}
\begin{center}
\includegraphics[scale=2]{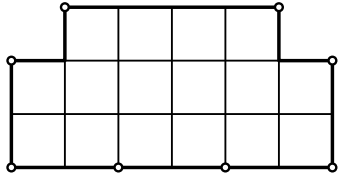}
\end{center}
\caption{\label{fig:one}A CAT(0) octagon}
\end{figure}

One corollary of this construction is that every sequence 
$(l_1,\ldots,l_n)$ with $p$ divisible by 4 and each $l_i$ strictly less than 
$p/8$ can be realized: the condition on the lengths $l_i$ guarantees
that if we choose $R$ to be a square, then the four rectangles $R'_j$
are disjoint.

Next consider the problem of constructing a CAT(0) $n$-gon $S$ 
in which all the curvature is concentrated at a single vertex.  
In the case when the vertex $v$ with non-zero curvature is in 
the interior of $S$, 
this is equivalent to identifying the sides of $n$ rectangles 
$R_1,\ldots, R_n$ where $R_i$ has side lengths $k_i$ and $k_{i+1}$.
(As usual the index $i$ should be assumed to lie in the integers
modulo $n$.)  This gives rise to a CAT(0) $n$-gon with side lengths
$l_i$ given by the equations $l_i=k_i+k_{i+2}$.  The cases when the
vertex with non-zero curvature is in $\partial S$ arise as degenerate
cases.  If $k_{i+1}=0$ but $k_j>0$ for $j\neq i$, then the vertex with
non-zero curvature lies in the interior of the $i$th side of the
$n$-gon.  If $k_i=k_{i+1}=0$ but $k_j>0$ for $j\neq i,i+1$, then the
curvature is concentrated at the corner $v_i$ of $S$.  Other
degenerate cases do not give rise to CAT(0) $n$-gons.

The equations $l_i=k_i+k_{i+2}$ can be solved uniquely for $k_i$ if 
and only if $n$ is not divisible by 4.  If $n= 4m$, there is a 
solution if and only if 
\[\sum_{i=1}^m l_{4i-3} = \sum_{i=1}^m l_{4i-1} 
\quad \mbox{and} 
\quad 
\sum_{i=1}^m l_{4i-2} = \sum_{i=1}^m l_{4i}.
\] 
If we define $m$ by $m=n$ in the case when $n$ is odd, and 
$m=n/2$ when $n$ is congruent to 2 modulo 4, then the unique 
solution in these cases is given by 
\[k_i = \frac{1}{2}\sum_{j=0}^{m-1} (-1)^j l_{i+2j}.
\]

At least in the case when 4 does not divide $n$, this solves 
completely the question of whether there exists a CAT(0) $n$-gon 
with curvature at just one vertex and side lengths $(l_1,\ldots,l_n)$.   
The above equation computes $k_1,\ldots,k_n$, and then 
there is such an $n$-gon if and only if each $k_i$ is a positive 
integer, and either at most one $k_i$ is equal to zero or there 
exists $i$ such that $k_j=0$ if and only if $j=i$ or 
$j=i+1$.  Since every CAT(0) pentagon has curvature at just one 
vertex, this completely solves the question of the possible side 
lengths of CAT(0) pentagons.  

Figure~\ref{fig:two} shows three CAT(0) hexagons that can be 
built in this way by identifying the sides of six Euclidean
rectangles.  

\begin{figure}
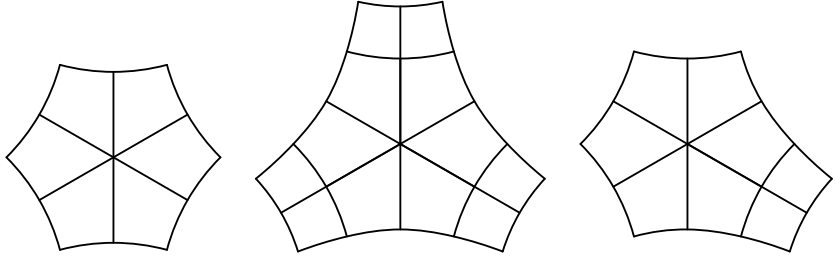

\begin{center}
\includegraphics[scale=1]{hexagon1.mps}
\quad\includegraphics[scale=1]{hexagon2.mps}
\quad\includegraphics[scale=1]{hexagon3.mps}
\end{center}
\caption{\label{fig:two}
Three CAT(0) hexagons}
\end{figure}



\section{Acyclic 2-complexes}\label{sec:acyc}

We give some constructions for acyclic locally CAT(0) 2-complexes.
In each case we start with a balanced finite group presentation.  
We take a rose in which each petal has the same length, and attach 
right-angled CAT(0) 2-cells as constructed in
Section~\ref{sec:polygons}.  Since we wish to build our right-angled
polygons using unit squares, the lengths of the petals of the roses 
that we use will be larger; petals of length 4 will suffice.

Fix integers $m$ and $n$ with $3\leq m\leq n$, and 
construct a 2-complex $Y$ by
attaching $2n$ polygons with $2m$ sides to a rose $Y^1$ with $2n$
petals.  The petals of the rose will be labelled $a_i$ and $b_i$ for
$i\in \zz/n$.  The $2m$ sides of the polygons will be labelled not by
a single letter, but instead by a longer word.  To describe the
process it is convenient to let $A_i$ stand for a reduced word in the
letters $\{a_1,\ldots,a_n\}$ and their inverses which begins and ends
with $a_i$.  For example, $a_1a_2^{-2}a_3a_1$ is a valid choice for
$A_1$, whereas $a_1^{-1}a_3^2a_1$ is not.  Similarly, let $B_i$ denote
a reduced word in the letters $\{b_1,\ldots,b_n\}$ and their inverses
which begins and ends with $b_i$.  Define a meeting point to be a pair
consisting of a word $a_i^{\epsilon_1}b_j^{\epsilon_2}$ and its formal
inverse $b_j^{-\epsilon_1}a_i^{-\epsilon_2}$, where $1\leq i,j\leq n$
and $\epsilon_1,\epsilon_2\in \{\pm 1\}$, so that there are $4n^2$
distinct meeting points.  The sides of the polygons will be labelled
by words $A_i$ and $B_j$ of the types described above, or their
inverses, in such a way that at the meeting points appearing at the
$4mn\leq 4n^2$ corners of the polygons are all distinct.  It is clear that
there are many ways to do this.  Two such labellings will suffice for
our purposes, the $2n$ words
\[ A_1B_jA_2B_{j+1}A_3\cdots A_mB_{j+m-1}\quad\hbox{and}\quad  
A_1B_j^{-1}A_2B_{j-1}^{-1}A_3\cdots A_mB_{j-m+1}\] 
for $j\in \zz/n$ or the $2n$ words 
\[A_1B_jA_2B_j\cdots A_mB_j\quad \hbox{and}\quad
A_1B_j^{-1}A_2B_j^{-1}A_3\cdots A_mB_j^{-1}\]
for $j\in \zz/n$.  In each case, the different appearances of
each $A_i$ and $B_j$ should be read as representing possibly 
different words of the allowed type, rather than lots of copies of the 
same word.  

The subwords $A_i$ and $B_j$ occurring in the $2n$ words that
represent the boundaries of the polygons should be chosen in 
such a way that there is a tesselated CAT(0) polygon whose 
side lengths are four times the word lengths $|A_i|$ and $|B_j|$.  
For example, if $m\geq 5$ and all of the words $A_i$ and $B_j$ are 
of approximately equal lengths, then such tesselated CAT(0) polygons
can always be made by cutting the corners off a Euclidean rectangle 
as described in the previous section.  

The 2-complex $Y$ is built by starting with a rose $Y^1$ with $2n$
petals, each built from four 1-cells.  Each petal is labelled by one of
the letters $a_i$ or $b_i$.  Now the $2n$ tesselated CAT(0) polygons
are attached using the $2n$ words described above.  To check that this
2-complex is locally CAT(0), we check the link of each vertex.  To
simplify the proof, we consider only the case when the curvature in 
each of the $2n$ polygons is concentrated at interior vertices.  The
links of vertices in the interiors of the polygons are circles of
circumference at least $2\pi$, and so cause no problems.  Similarly,
the link of any vertex in $Y^1$, apart from the central vertex of the
rose, consists of two points joined by a number of arcs each of length
$\pi$.  It remains to consider the central vertex of the rose.  The
link of this vertex is a graph with $4n$ vertices, the inward and
outward ends of the $2n$ petals of the rose, which we shall denote by
$a_i^{(i)}$, $a_i^{(o)}$, $b_i^{(i)}$ and $b_i^{(o)}$ for $i\in
\zz/n$.  There are $4mn$ edges in this graph of length
$\pi/2$, coming from the corners of the $2n$-gons, and a large number
of edges of length $\pi$, coming from the interiors of the sides of
the $2n$-gons.  For each $i,j\in \zz/n$ and each $x,y\in \{i,o\}$,
there is at most one short edge between $a_i^{(x)}$ and $b_j^{(y)}$.
Thus the short edges form a subgraph of the complete 
bipartite graph $K^{2n,2n}$.
Each long edge either has both of its vertices in $\{a_i^{(x)}: i\in
\zz/n,\,\, x\in \{i,o\}\}$ or both of its vertices in $\{b_i^{(x)}:
i\in \zz/n,\,\, x\in \{i,o\}\}$.  The fact that the words along the
sides of the $2n$-gons are reduced implies that no single long edge is
a loop.  Hence the shortest loops in the link graph consist of either
two long edges, one long and two short edges, or four short edges.  It
follows that the 2-complex $Y$ is locally CAT(0).  

The 2-complex $Y$ will be acyclic if and only if the abelianization of
its fundamental group is trivial.  This of course places extra
conditions on the words along the edges of the $2n$-gons.  We state 
some results that can be proved using this process.

\begin{remark} 
Note that the subcomplex of $Y$ consisting of the 0-cell and the 
1-cells labelled `$a_i$' is a totally geodesic subcomplex, since 
the points $a_i^{(x)}$ in the vertex link in $Y$ are separated by 
at least $\pi$.  It follows that the subgroup of $\pi_1(Y)$ generated
by $a_1,\ldots, a_n$ is a free group on these $n$ generators.  
Similarly, the subgroup of $\pi_1(Y)$ generated by $b_1,\ldots,b_n$ 
is a free group freely generated by these elements.  
\end{remark} 

\begin{proposition} 
\label{prop:acycone} 
The group presented  on generators $a,\ldots,f$ subject to the
following 6-relators is non-trivial, torsion-free and acyclic.  The
corresponding presentation 2-complex may be realized as a 
non-positively curved  square complex (or 
as a negatively curved polygonal complex).  
\[abcdef,\qquad 
ab^{-1}c^2f^{-1}e^2d^{-1}, \qquad 
a^2fc^2bed,\]
\[ad^{-2}cb^{-2}ef^{-1},\qquad 
ad^2cf^2eb^2,\qquad 
af^{-2}cd^{-1}eb^{-2}.  
\]
\end{proposition}

\begin{proof} 
Here $n=3$, and we have written $a,c,e$ in place of $a_1,a_2,a_3$ 
and similarly $b,d,f$ in place of $b_1,b_2,b_3$.  Each word $A_i$
and $B_j$ consists of a power of a single letter.  Using
techniques from Section~\ref{sec:polygons}, one shows easily that
CAT(0) hexagons with side lengths $(4,4,4,4,4,4)$, $(8,4,8,4,4,4)$ 
and $(8,4,8,4,8,4)$ can be built from unit squares, and these suffice 
to make this 2-complex: in fact, the cubical subdivisions of the 
three hexagons shown in Figure~\ref{fig:two} can be used for this 
purpose.  The claim 
in parentheses follows from the fact that obtuse hexagons with 
side lengths in these ratios can be realized in the hyperbolic 
plane.  
\end{proof}

\begin{proposition}\label{prop:acyctwo} 
There is a non-contractible finite locally CAT(0) acyclic
2-dimensional cubical complex $Y$ which admits a cubical involution 
$\tau:Y\rightarrow Y$ such that 
\begin{enumerate}
\item The fixed point set $Y^\tau$ consists of a single point;
\item The quotient space $Y/\tau$ is contractible;
\item $Y$ contains an isometrically embedded $\tau$-invariant
  2-petalled rose with $\tau$ acting by swapping the two petals.
\end{enumerate} 
\end{proposition} 

\begin{proof} 
Let $n=4$, and for $i\in \zz/n$ define 
\[A_i=a_ia_{i+2}a_i^{-2}a_{i+2}^{-1}a_i, \qquad
B_i=b_ib_{i+2}b_i^{-2}b_{i+2}^{-1}b_i.\] 
These words $A_i$ and $B_i$ will be fixed throughout this proof.  
Now consider the $2n$ words 
\[a_iA_iB_iA_{i+1}B_iA_{i+2}B_iA_{i+3}B_i\quad \hbox{and}\quad
b_iB_iA_i^{-1}B_iA_{i+1}^{-1}B_iA_{i+2}^{-1}B_iA_{i+3}^{-1}.\]
The techniques of Section~\ref{sec:polygons} enable one to construct a
CAT(0) octagon with each side length 24 except for one side of length
28.  For example, one could take the CAT(0) tesselated octagon
depicted in Figure~\ref{fig:one}, and collar it five times to produce
a CAT(0) octagon with seven sides of length 12 and one of length
14.  The cubical subdivision of this octagon is the octagon that we 
require.  Attach 8 copies of this octagon to an 8-petalled rose (with
petals labelled by the 8 generators) to make the 2-complex $Y$.  

There is an action of $\zz/4$ on the eight generators, with 
$1\in \zz/4$ acting by sending $a_i$ to $a_{i+1}$ and $b_i$ to
$b_{i+1}$.  This extends to a free action of $\zz/4$ on the 
eight octagons of $Y$.  We define the involution $\tau$ to 
be the action of the element $2\in \zz/4$, so that 
$\tau.a_i=a_{i+2}$ and $\tau.b_i=b_{i+2}$.  The fixed point 
set for the action of $\tau$ on $Y$ is just the central vertex
of the rose, and so is a single point as claimed.  Let $p:Y
\rightarrow Y/\tau$ be the quotient map.  There is a natural 
cell structure on $Y/\tau$ with one 0-cell and four 1-~and 
2-cells.  Since $p(a_i)=p(a_{i+2})$ and $p(b_i)=p(b_{i+2})$, 
the words describing the attaching maps for the four 2-cells 
are no longer reduced.  After reduction they are equal to 
$p(a_1)$, $p(a_2)$, $p(b_1)$ and $p(b_2)$.  Thus the 2-complex
$Y/\tau$ is homotopy equivalent to a wedge of four 2-discs, 
and so is contractible.  

The subcomplex of the rose consisting of the base vertex and 
the petals labelled $a_0$ and $a_2$ is a 2-petalled $\tau$-invariant 
rose which is isometrically embedded in $Y$.  
\end{proof}  

\begin{remark} 
Here are two other families of group presentations that can be 
shown to arise as fundamental groups of finite locally CAT(0)
square complexes using similar techniques: 
\begin{itemize}
\item Generators $a,\ldots, f$, and relators 
\[abcb^{-1}eb,\quad adcd^{-1}ed,\quad afcf^{-1}ef\] 
\[ab^{-1}cb^2eb^{-1},\quad a^2d^{-1}cded^{-1},\quad
af^{-1}c^2fef^{-1}.\] 
\item Fix $n\geq 8$, and take generators $a_i$ and relators
  $a_ia_{i+1}a_{i+3}^{-1}a_{i+1}^{-1}a_{i+3}$, for $i\in\zz/n$.  
\end{itemize} 
\end{remark}

\section{Contractible and acyclic complexes} 
\label{sec:conacyc}

For our main result we require a pair of locally CAT(0) cubical 
complexes $(A',A)$ and a cubical involution $\tau$ of the pair 
with the following properties: 
\begin{enumerate} 
\item $A$ and $A'$ are acyclic
\item $A$ is not contractible 
\item $A$ is a totally geodesic subcomplex of $A'$
\item $A$ is contained in the fixed point set for $\tau$ acting on $A'$ 
\item $A'/\tau$ is contractible 
\end{enumerate}

The following proposition gives a fairly simple construction 
that has almost all of the properties that we require.  

\begin{proposition} \label{prop:acyclem}
Let $A$ be an acyclic CW-complex, and let $\tau$ denote the involution
of $A\times A$ defined by $\tau(a,b)=(b,a)$.  The quotient space 
$C=(A\times A)/\langle \tau\rangle$ is contractible.  
\end{proposition} 

\begin{proof} It suffices to show that the fundamental group
$\pi_1(C)$ is trivial and that $C$ is acyclic.  Since $\tau$ 
fixes points on the diagonal of $A\times A$, $\pi_1(C)$ is isomorphic 
to the quotient of the wreath product 
$\pi_1(A)\wreath \langle \tau\rangle$ by the normal
subgroup generated by $\tau$.  Thus we need to show that this
wreath product is generated by conjugates of $\tau$.  Recall that 
elements of the wreath product may be written in the form
$(\alpha,\beta,\tau^\epsilon)$, where $\alpha,\beta\in \pi_1(A)$ and  
$\epsilon = 0$ or $1$.  The product of two such elements is given by
\[(\alpha,\beta,1)(\gamma,\delta,\tau^\epsilon)=
(\alpha\gamma,\beta\delta,\tau^\epsilon), \qquad 
(\alpha,\beta,\tau)(\gamma,\delta,\tau^\epsilon)=
(\alpha\delta,\beta\gamma, \tau^{1-\epsilon}).\]
The inverse of $(\alpha,1,1)$ is $(\alpha^{-1},1,1)$, and
so for any $\alpha$, 
\[(\alpha,\alpha^{-1},\tau)= (\alpha,1,1)(1,1,\tau)(\alpha^{-1},1,1)\]
is a conjugate of $\tau$, and $(\alpha,\alpha^{-1},1)$ is a product of
conjugates of $\tau$.  In particular, for any $\beta$ and $\gamma$, 
\[(\beta^{-1}\gamma^{-1}\beta\gamma,1,1)= 
(\beta^{-1},\beta,1)(\gamma^{-1}\beta,\beta^{-1}\gamma,1) 
(\gamma,\gamma^{-1},1)\]
lies in the normal subgroup generated by $\tau$.  But since $\pi_1(A)$
is perfect, an arbitary element of $\pi_1(A)$ may be expressed as a
product of such commutators.  Hence for any $\alpha$, $(\alpha,1,1)$ and 
$(1,\alpha,1)=(1,1,\tau)(\alpha,1,1)(1,1,\tau)$ lie in the normal
subgroup generated by $\tau$, which is therefore the whole of
$\pi_1(A)\wreath \langle \tau\rangle$.  

Now let $\Delta$ denote the diagonal copy of $A$ within $A\times A$.  
Note also that $\Delta$ is equal to the fixed point set for the action
of $\tau$ on $A\times A$.  
Since $A$ is acyclic, the inclusion of $\Delta$ in $A\times A$ induces
a homology isomorphism.  Hence the relative homology groups 
$H_*(A\times A,\Delta)$ are all trivial.  It follows that the relative
cellular chain complex $C_*(A\times A,\Delta)$ is an exact sequence of 
free $\zz\langle \tau\rangle$-modules, and hence is split.  Thus this
sequence remains exact upon tensoring over $\zz\langle \tau\rangle$
with $\zz$ (with $\tau$ acting trivially).  Now there is a natural
isomorphism 
\[C_*(C,\Delta)= C_*((A\times A)/\langle \tau\rangle,\Delta)\cong 
C_*(A\times A,\Delta)\otimes_{\zz\langle \tau\rangle}\zz,\]
and hence the relative homology groups $H_*(C,\Delta)$ are trivial.  
From this it follows that the inclusion of $\Delta$ in $C$ is a homology
isomorphism, and hence that $C$ is acyclic, as required.  
\end{proof} 

\begin{remark} 
For any prime $p$, a similar argument may be used to show that 
$A^p/C_p$ is contractible, where $C_p$ acts on $A^p$ by freely
permuting the factors.  
\end{remark} 

The only problem with the above construction is that
if $A$ is an acyclic locally CAT(0) cubical complex,  
the diagonal copy of $A$ inside $A\times A$ is not a 
subcomplex.  Nevertheless, the barycentric subdivision 
of the diagonal copy of $A$ is a subcomplex of the 
barycentric subdivision of $A\times A$.  

Proposition~\ref{prop:acyclem} can be used to prove a weaker version of 
Theorem~A, in which $T_X$ is not a locally CAT(0) cubical 
complex, but instead a locally CAT(0) polyhedral complex.  
In the non-metric context, Proposition~\ref{prop:acyclem} 
is also useful.  It can be used to give an alternative 
proof of the main theorem of~\cite{LearyNucinkis}, starting 
from a proof of the Kan-Thurston theorem along the lines 
of that of Baumslag-Dyer-Heller~\cite{BDH}.

\begin{theorem}\label{thm:aaprime}
There is a pair $(A',A)$ of locally CAT(0) cubical complexes 
equipped with an involution $\tau$ having the properties listed at the
start of this section.  Furthermore $A$ is the fixed point set for 
the action of $\tau$.  The complex $A$ is 2-dimensional and $A'$ is 
3-dimensional.  
\end{theorem} 

\begin{proof} 
Let $Y$ be the 2-complex constructed in
Proposition~\ref{prop:acyctwo}, and let $\tau'$ denote the involution
of $Y$ which is denoted `$\tau$' in Proposition~\ref{prop:acyctwo}.  
Let $X_0$ be the direct product $Y\times [-4,4]$ of $Y$ with an 
interval of length 8, and define an involution $\tau_0$ on $X_0$ by 
$\tau_0(y,t)= (\tau'(y),-t)$.  By definition, $X_0$ can be viewed as being 
built from unit cubes.  Define $X_1$ to be the mapping torus of
$\tau':Y\rightarrow Y$, i.e., the quotient of $X_0$ obtained by 
identifying $(y,4)$ and $(\tau'(y),-4)$ for all $y\in Y$.  Since 
$\tau'$ is an involution, it follows readily that $\tau_0$ defines 
an involution $\tau_1$ on $X_1$.  Let $y_0$ denote the unique 0-cell 
in $Y$, and define $X_2$ by identifying the images of the points 
$(y_0,0)$ and $(y_0,4)$ in $X_1$.  As before, $\tau_1$ passes to an 
involution $\tau_2$ on $X_2$.  Each of $X_0$, $X_1$ and $X_2$ is a 
locally CAT(0) cubical complex, as may be seen by checking that 
all vertex links are flag.  (Every vertex link in $X_0$ or $X_1$ 
is isomorphic either the cone on or the suspension of a vertex 
link in $Y$, and the `new' link in $X_2$ is isomorphic to the 
disjoint union of two links from $X_1$.)  

\begin{figure}
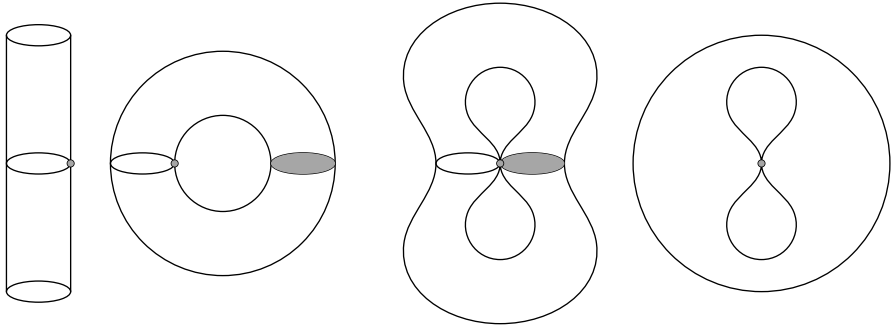

\begin{center}
\includegraphics[scale=1.2]{exx0.mps}
\quad\includegraphics[scale=1.2]{exx1.mps}
\quad\includegraphics[scale=1.2]{exx2.mps}
\quad\includegraphics[scale=1.2]{exx3.mps}
\end{center}
\caption{\label{fig:three}
The spaces $X_0$, $X_1$, $X_2$ and $X_3$, with the fixed point sets
for the involutions shaded in grey.}
\end{figure}

Let $Z_2$ be image in $X_2$ of $\{y_0\}\times [-4,4]$, so that 
$Z_2$ is a 2-petalled rose in which each petal has length 4.  
The subspace $Z_2$ is totally geodesic and invariant under the 
action of $\tau$ (which swaps the two petals).  Let $X_3$ be a 
disjoint copy of $Y$, and let $Z_3$ be 
the subspace of $X_3$ consisting of the 0-cell $y_0$ and the 
two 1-cells labelled $a_1$ and $a_3$.  Then $Z_3$ is also a 
2-petalled rose in which each petal has length 4, $Z_3$ is 
totally geodesic in $Y$ and is invariant under the action 
of the involution $\tau'$, which swaps the two petals.  Let 
$\phi:Z_2\rightarrow Z_3$ be an isometry which is equivariant 
for the two given involutions, and define $A'$ by taking the 
disjoint union of $X_2$ and $X_3$ and identifying $z$ with 
$\phi(z)$ for all $z\in Z_2$.  Since $A'$ was constructed from 
two locally CAT(0) cubical complexes by identifying isometric 
totally geodesic subcomplexes, it follows that $A'$ is itself 
a locally CAT(0) cubical complex.  (Alternatively, one may 
check the structure of vertex links in $A'$.)  Since $\phi$ 
has the property that $\phi\tau_2(z)= \tau'\phi(z)$ for all 
$z\in Z_2$, one may define an involution $\tau$ on $A'$ by 
$\tau(a)= \tau_2(a)$ for $a\in X_2$ and $\tau(a)=\tau'(a)$ 
if $a\in X_3$.  Let $A$ be the fixed point set for $\tau$.  
It is readily seen that $A$ is equal to the image of $Y\times\{0\}$
inside $A'$, which is isometric to $Y$ and so is acyclic.  
The spaces $X_0$, $X_1$, $X_2$ and $X_3$ are depicted in 
Figure~\ref{fig:three}.  The fixed point set for the involution on
each $X_i$ is shaded in grey.  

It remains to check that $A'$ is acyclic and $A'/\tau$ is
contractible.  Since $Y$ is acyclic, it follows easily that 
$X_1$ has the same homology as a circle.  
A Mayer-Vietoris argument now shows that the 
inclusion map $Z_2\rightarrow X_2$ is a homology isomorphism.  
Since $X_3$ is acyclic, another Mayer-Vietoris argument 
shows that $A'$ is acyclic.  

Note that $X_1/\tau_1$ is homeomorphic to 
the mapping cylinder of the quotient map $Y\rightarrow Y/\tau'$.  Since 
$Y/\tau'$ is contractible, it follows that $X_1/\tau_1$ is 
contractible.  From this it follows that $X_2/\tau_2$ is homotopy 
equivalent to a circle, and that the inclusion map
$Z_2/\tau_2\rightarrow X_2/\tau_2$ is a homotopy equivalence.  
Since $X_3/\tau'$ is contractible, it follows from van Kampen's 
theorem that $A'/\tau'$ has trivial fundamental group and it 
follows from the Mayer-Vietoris theorem that $A'/\tau'$ is 
acyclic.  Hence $A'/\tau'$ is contractible as claimed.  
\end{proof}

\section{The main result} 
\label{sec:main}

Our proof is close to Maunder's proof of the Kan-Thurston
theorem~\cite{Maunder}.  However, we use a slightly wider class of
spaces than Maunder's class of `ordered simplicial complexes'.  We
start with the category of $\Delta$-complexes, or semi-simplicial
sets~\cite{Hatcher,RourkeSanderson}.  Our construction will apply to
the full subcategory $\calc$ of $\Delta$-complexes such that the
${\binom{n+1}2}$ edges of each $n$-simplex are distinct.

Note that the barycentric subdivision of any $\Delta$-complex has the
stronger property that the $(n+1)$ vertices of each $n$-simplex are
distinct, and so the barycentric subdivision of any $\Delta$-complex
is in $\calc$.  Similarly, the barycentric subdivision $K'$ of any
simplicial complex $K$ is naturally an object of $\calc$, and any map
$f:K\rightarrow L$ of simplicial complexes that is injective on each
simplex induces a map of $\Delta$-complexes from $K'$ to $L'$.

Let $X$ and $Y$ denote $\Delta$-complexes, and let $x$ be a vertex of
$X$.  The link of $x$, denoted $\link_X(x)$, is another
$\Delta$-complex.  Any map of $\Delta$-complexes $f:X\rightarrow Y$
induces a map $f_x:\link_X(x)\rightarrow\link_Y(f(x))$.  (This is one
advantage of $\Delta$-complexes over simplicial complexes.)  The map
$f$ is said to be \emph{locally injective} if for each vertex $x$, 
the map $f_x$ is injective.  It can be shown that a map of
$\Delta$-complexes is locally injective if and only if the induced map
of topological realizations is locally injective.

 For each $X$ in $\calc$, we will construct a locally CAT(0) cubical
complex $T_X$, a map $t_X:T_X\rightarrow X$ and an involution $\tau$
on $T_X$ having the properties stated in Theorem~A.  The
construction will be natural for any map $f:X\rightarrow Y$ in
$\calc$; however the induced map $T_f:T_X\rightarrow T_Y$ will be
locally isometric only in the case when $f$ is locally injective.
The following statement is a summary of what we will prove.  

\begin{theorem}\label{thm:deltaversion}
Let $\cald$ denote the category whose objects are 
locally CAT(0) cubical complexes equipped with a cellular 
isometric involution $\tau$ and whose morphisms are $\tau$-equivariant 
cubical maps.  There is a functor $T$ from $\calc$ to $\cald$ which 
has all of the properties listed in Theorem~A.  
\end{theorem}

Theorem~A follows from Theorem~\ref{thm:deltaversion} 
by composing $T$ with the barycentric subdivision functor from the 
category of simplicial complexes and simplicial maps that are
injective on each simplex to the category $\calc$.  

Before beginning the proof of Theorem~\ref{thm:deltaversion}, we 
need two more pieces of
notation.  If $f:T\rightarrow U$ is a locally isometric cellular 
map of cubical complexes, and $n$ is a positive integer, let 
$M(n,f)$ denote the mapping cylinder of $f$ of length $n$, i.e., 
the cubical complex obtained from $T\times [0,n]\coprod U$ by 
identifying, for each $t\in T$, the points $(t,n)$ and $f(t)$.  

Finally, let $(A',A)$ be the pair of spaces constructed in
Theorem~\ref{thm:aaprime}, and let $j$ denote a fixed locally 
isometric closed loop $j:[0,4]\rightarrow A$.  For example, 
$j$ could be the map that goes at constant speed around the 
generator $a_1$ for the fundamental group of $A$.  (Recall 
that $A$ is isometric to the cubical complex $Y$ of
Proposition~\ref{prop:acyctwo}.)

\begin{proof} 
For each finite $\Delta$-complex $X$ in $\calc$, 
we inductively define $T_X$, the map $t_X:T_X\rightarrow X$, 
another locally CAT(0) cubical complex $U_X$, a locally isometric 
cubical map $i_X:T_X\rightarrow U_X$ and cubical involutions $\tau$
on $T_X$, $U_X$, so that all of the following hold.  
\begin{itemize}

\item The map $t_X$ induces integral homology isomorphisms 
$T_X\rightarrow X$ and $T_X^\tau\rightarrow X$ and for any 
vertex $x_0$ of $X$, an isomorphism of fundamental groups
$\pi_1(T_X/\langle\tau\rangle, T_{x_0})\rightarrow \pi_1(X,x_0)$.  

\item $U_X$ is acyclic and $U_X/\tau$ is contractible.  

\item The map $i_X$ is $\tau$-equivariant.  

\item $T_X$ and $U_X$ are finite.  

\item $\dim(T_X)=\dim(X)$, except that if $\dim(X)=2$ then
  $\dim(T_X)=3$.  

\item $\dim(U_X)=\dim(X)+1$ except that $\dim(U_X)=3$ if $\dim(X)=1$, 
and $\dim(U_X)=0$ if $X$ is 0-dimensional.  

\item If $W$ is a subcomplex of $X$, then $T_W$ is a totally geodesic
  subcomplex of $T_X$ and $U_W$ is a totally geodesic subcomplex of
  $U_X$, and these inclusions are $\tau$-equivariant.  

\end{itemize}  

In the case when $X$ is 0-dimensional, define $T_X$ to be $X$, 
define $U_X$ to be a single point, take $i_X:T_X\rightarrow U_X$ 
to be the unique map, and let $\tau$ be the trivial involution.  

In the case when $X$ is 1-dimensional, take $T_X$ to be the second 
barycentric subdivision of $X$, so that for each edge $x$ of $X$, 
$T_X$ is either a loop or a path consisting of 4 edges of length 1.  
Let $(A',A)$ be the pair of spaces constructed in
Theorem~\ref{thm:aaprime}, and for $U_X$ take the 1-point union of 
copies of the space $A'$ indexed by the 1-cells of $X$, 
with the involution $\tau$ as defined 
in Theorem~\ref{thm:aaprime}.  The map $i_X:T_X\rightarrow U_X$ is defined
by taking a copy of the map $j:[0,4]\rightarrow A\subseteq A'$ for 
each 1-cell.  

Now suppose that $T_X$, $U_X$, $\tau$ and $i_X$ have been defined
and have the above listed properties 
whenever $X\in \calc$ is finite and consists of at most $N-1$ 
simplices of various dimensions.  
Suppose now that $X$ is obtained from $W$ by adding a single 
$n$-simplex $\sigma$, for some $n\geq 2$.  
Let $\partial\sigma$ denote the subcomplex
of $W$ consisting of all the simplices contained in the boundary 
of $\sigma$.  To simplify notation slightly, let $U'$ denote
$U_{\partial\sigma}$, let $T'$ denote the $\tau$-fixed point set 
$T_{\partial\sigma}^\tau$, and let $i':T'\rightarrow U'$ denote the 
restriction to $T'$ of the map $i_{\partial\sigma}$.  Now define 
\[T_X = T_{W} \cup_{T'} M(4,i':T'\rightarrow U'),\]
\[U_X = U_W \cup_{T'} (T'\times A)\cup_{A} A'.\] 
More precisely, let $x_0$ be the initial vertex of $\sigma$, and let 
$w_0= T_{x_0}$ be the corresponding vertex of $T_\sigma$.  Note that 
$w_0$ is fixed by the involution $\tau$ on $T_\sigma$.  Also let
$a_0\in A$ be the fixed point for the involution of $A$ described in 
Proposition~\ref{prop:acyctwo}.  Now the space $U_X$ is obtained 
from the disjoint union $U_X = U_W \coprod (T'\times A)\coprod A'$ 
by identifying $T'\times\{a_0\}\subseteq T'\times A$ with $T'\subseteq
U_W$ and identifying $\{w_0\}\times A$ with $A\subseteq A'$.  

The involution $\tau$ on each of $T_X$ and $U_X$ is defined using the 
given involutions on $T_W$, $U_W$, $T'$, $U'$ and $A'$.  The map 
$i_X:T_X\rightarrow U_X$ is defined to equal $i_W$ on $T_W$.   
On $U'\subseteq M(4,i')$ the map 
is equal to $i_{\partial\sigma}$, and on the rest of the mapping
cylinder $M(4,i')$ it is induced by the map $(s,t)\mapsto (t,j(s))$ 
from $[0,4]\times T'$ to $T'\times A$.  
The map $t_X$ is defined by $t_W$ on $T_W\subseteq T_X$, by taking 
$U' \subseteq M(4,i')$ to the barycentre of $\sigma$, and on 
the rest of the mapping cylinder $M(4,i)$ by linear interpolation 
between the map $t_W|_{T'}$ and the constant map that sends $T'$ to 
the barycentre of $\sigma$.  

It follows easily that the maps $T_W\rightarrow T_X$ and
$U_W\rightarrow U_X$ are isometric embeddings of totally geodesic
subcomplexes.  The quickest way to see this is by induction, 
using a gluing lemma, but
one could also show directly that each vertex link in $T_W$
(resp.\ $U_W$) is a full subcomplex of the corresponding vertex link in
$T_X$ (resp.\ $U_X$).  Similarly, an induction shows that the map 
$i_X:T_X\rightarrow U_X$ is a $\tau$-equivariant locally isometric
cubical map.  

It is easily checked that the map $t_X$ induces isomorphisms from 
each of $H_*(T_X,T_W)$, $H_*(T_X^\tau,T_W^\tau)$ and
$H_*(T_X/\tau,T_W/\tau)$ to $H_*(X,W)$.  
By the 5-lemma and induction it follows that $t_X$ induces
isomorphisms from $H_*(T_X)$, $H_*(T_X^\tau)$ and $H_*(T_X/\tau)$ to
$H_*(X)$.   

By induction, the fundamental group of $U'/\tau$ is trivial, and $t_W$
induces an isomorphism $\pi_1(T_W/\tau,T_{w_0})\rightarrow
\pi_1(W,w_0)$, for any vertex $w_0\in W$.  By Van Kampen's theorem, 
it follows that $t_W$ induces an isomorphism from 
$\pi_1(T_X/\tau,T_{x_0})$ to $ \pi_1(X,x_0)$ as claimed.  

The Mayer-Vietoris theorem shows that $H_*(U_X,U_W)=\{0\}$,
and so by induction $U_X$ is acyclic.  Similarly, Van Kampen's theorem
and an induction shows that $U_X$ is simply-connected.  It follows
that $U_X$ is contractible as required.  
This completes the proof that the construction applies to any finite
$X$ and has the properties listed at the start of the proof in this 
case.  

For the general case, define $T_X$, $t_X$ and $\tau$  as direct 
limits over the finite subcomplexes of $X$.  (We shall not need $U_X$ or
$i_X$ in the case when $X$ is infinite, but these could also be 
defined in this way for arbitrary $X$.)  It is immediate from the
definition that in the general case, $t_X$ induces isomorphisms 
from $H_*(T_X)$, $H_*(T_X^\tau)$ and $H_*(T_X/\tau)$ to $H_*(X)$ and 
an isomorphism, for any $x_0$, from $\pi_1(T_X,T_{x_0})$ to
$\pi_1(X,x_0)$.  

To establish the metric properties of $T_X$ in the general case, it is
easiest to use Gromov's criterion (Theorem~\ref{gromovthm}).  If $t$ is
a vertex of $T_X$, then the link of $t$, $\link_{T_X}(t)$, is the
limit of the $\link_{T_Y}(t)$, where $Y$ ranges over the finite
subcomplexes of $X$ such that $T_Y$ contains $t$.  Since each such
$\link_{T_Y}(t)$ is a flag complex, it follows that $\link_{T_X}(t)$
is a flag complex.  Hence by Theorem~\ref{gromovthm}, $T_X$ is locally
CAT(0).  Now suppose that $W$ is a subcomplex of $X$, let $t$ be a
vertex of $T_W$, and let $e_0,\ldots,e_n$ be vertices of
$\link_{T_W}(t)$.  Let $U$ be any finite subcomplex of $W$ such that
$t$ is a vertex of $T_U$ and such that each of $e_0,\ldots,e_n$ is a
vertex of $\link_{T_U}(t)$.  If $Y$ is any finite subcomplex of $X$
that contains $U$, then $\link_{T_U}(t)$ is a full subcomplex of
$\link_{T_Y}(t)$.  It follows that every simplex of $\link_{T_X}(t)$
whose vertex set is contained in $\{e_0,\ldots,e_n\}$ is already a
simplex in $\link_{T_U}(t)$.  Hence $\link_{T_W}(t)$ is a full
subcomplex of $\link_{T_X}(t)$, and so the inclusion $T_W\rightarrow
T_X$ is totally geodesic.  This completes the construction of $T_X$
and $t_X$ having all the properties listed at the start of the proof.

To complete the proof of Theorem~\ref{thm:deltaversion}, it remains to
establish that $T_X$, $t_X$ and $\tau$ have the rest of the properties 
listed in the statement of Theorem~A{}.  We first
consider the three listed naturality properties.  The fact that 
any map of $\Delta$-complexes $f:X\rightarrow Y$ induces a map of
cubical complexes $T_f:T_X\rightarrow T_Y$ is clear from the
construction.  The fact that $T_f$ embeds $T_X$ isometrically as a 
totally geodesic subcomplex in the case when $f$ is injective 
follows immediately from the case when $f$ is the inclusion of a 
subcomplex, which was proved during the construction of $T_X$.  
The fact that $T_f$ is locally isometric when $f$ is locally injective 
now follows from Theorem~\ref{isom}, since if $f$ is locally injective 
then for any vertex $t$ of $T_X$, the induced map from
$\link_{T_X}(t)$ to $\link_{T_Y}(f(t))$ is the inclusion of a full
subcomplex in a flag complex.  

In the case when $X$ is simply-connected, properties 1--3 listed in
the statement of Theorem~A follow from the properties
that we have already established.  In the case when $X$ is connected, 
but not simply connected, let $\wtX$ be the universal cover of $X$, 
and let $G$ be the fundamental group of $X$.  By naturality of the 
construction, $G$ acts by deck transformations on $T_{\wtX}$.  
The covering map $\wtX\rightarrow X$ induces a map 
$T_{\wtX}\rightarrow T_X$, which commutes with the action of $G$ on 
$T_{\wtX}$.  It follows that there is a $G$-equivariant isomorphism 
from $T_{\wtX}$ to the regular cover of $T_X$ with fundamental group 
equal to the kernel of the map from $\pi_1(T_X)$ to $G$.  Note also 
that the map $T_{\wtX}\rightarrow \wtX$ induces a $G$-equivariant 
homotopy equivalence $T_{\wtX}/\tau\rightarrow \wtX$.  It follows that
in the case when $X$ is connected, the map $T_X/\tau\rightarrow X$ 
is a homotopy
equivalence.  For each $n\geq 0$, let $\wtX^n$ denote the $n$-skeleton
of the universal cover of $X$, and let $\wtX^{-1}$ denote the empty
space.  By what we have proved already, for each $n\geq -1$ the map 
induced by $t_X$ from $H_*(T_{\wtX^{n+1}},T_{\wtX^n})$
to $H_*(X^{n+1},X^n)$ is a $G$-equivariant isomorphism.  It follows by 
induction that for any local coefficient system $M$ on $X$ and for 
any $n\geq 0$, the map $t$ induces an isomorphism 
$H_*(T_{X^n};M)\rightarrow H_*(X^n;M)$.  This establishes property~1
in the case when $X$ is connected and finite-dimensional.  The case 
when $X$ is connected follows from this case, since the inclusion 
of $X^n$ into $X$ induces an isomorphism on $H_i(-;M)$ for any $i<n$.  

The proof of property~3 in the case when $X$ is connected is very 
similar to the proof of property~1 in the connected case.  Properties 
1--3 in the case of general $X$ follow immediately from the connected
case.  
\end{proof} 

\section{Some other properties of $T_X$} 
\label{sec:sev}

\begin{theorem} 
For any odd prime $p$, there is a cubical complex 
$T_{p,X}$ which has similar properties to those described in the 
statement of Theorem~A, 
except that the involution $\tau$ of $T_X$ is replaced by an order 
$p$ cellular isometry $\tau_p$ of $T_{p,X}$.  
\end{theorem} 

\begin{proof}
It suffices to prove the $p$-analogue of
Theorem~\ref{thm:deltaversion}.  The given proof of
Theorem~\ref{thm:deltaversion} will apply, provided that we establish
a suitable $p$-analogue of Theorem~\ref{thm:aaprime}.  For this we 
need a $p$-analogue of Proposition~\ref{prop:acyctwo}.  In fact, 
if one puts $n=2p$ and uses the $2n$ words described exactly as 
in the proof of Proposition~\ref{prop:acyctwo}, one obtains an 
acyclic 2-complex $Y_p$ consisting of $4p$ octagons attached to a 
$4p$-petalled rose and admitting a free $\zz/{2p}$-action.  If 
$\tau_p$ is an order $p$ element of $\zz/{2p}$, then the action 
of $\tau_p$ on $Y_p$ has the required properties.  

The $p$-analogue of Theorem~\ref{thm:aaprime} is fairly similar to 
the original, but the definition of $X_{1,p}$, the $p$-analogue of 
the space $X_1$ appearing in the proof, may require some clarification.  
Let $Y'_p$ denote a second copy of $Y_p$, but with the trivial
$\zz/p$-action instead of the action described above, and let $Y''_p$ be
a disjoint union of $p$ copies of $Y'_p$, with $\zz/p$ acting by freely
permuting the copies.  (One could view $Y''_p$ as the direct product 
$Y'_p\times \zz/p$ with the diagonal action of $\zz/p$.)  There are 
two $\zz/p$-equivariant covering maps $f:Y''_p\rightarrow Y'_p$ and 
$g:Y''_p\rightarrow Y_p$.  The $\zz/p$-space $X_{1,p}$ is defined to
be the double mapping cylinder of $f$ and $g$: 
\[ X_{1,p}= \left(\textstyle{Y'\coprod Y''\times [0,4] \coprod
Y}\right)/ (y,0)\sim f(y),\,\,\, (y,4)\sim g(y)\quad \hbox{for all
  $y\in Y''$}. 
 \]
The remainder of the proof carries over easily to provide a pair 
$(A'_p,A_p)$ of locally CAT(0) cubical complexes with the required
properties.  
\end{proof}

\begin{proposition} The cubical complex $T_X$ is the cubical
  subdivision of a cube complex $\overline T_X$ which has 
  similar properties. 
\end{proposition}

\begin{proof} 
This involves checking that each of the cubical complexes used in the
construction can be viewed as the cubical subdivision of a cube
complex.  The octagons used in the construction of $Y$ in
Proposition~\ref{prop:acyctwo} can be chosen to be cubical
subdivisions, in which case $Y$ is the cubical subdivision of 
a cube complex $\overline Y$.  From this it follows that the 
cubical complexes $A$ and $A'$ 
constructed in Theorem~\ref{thm:aaprime} are the cubical subdivisions
of cube complexes $\overline A$ and $\overline A'$.  Now a cube 
complex $\overline T_X$ whose cubical subdivision is $T_X$ can 
be constructed from $\overline A$ and $\overline A'$ exactly as 
in the proof of Theorem~\ref{thm:deltaversion}.  
\end{proof}

\begin{proposition} 
For any $X$, the map $t_X:T_X\rightarrow X$ has the property that 
for any subcomplex $Y$ of $X$, $t_X^{-1}(Y)=T_Y$.  In particular, 
the map $t_X$ is proper.  
\end{proposition} 

\begin{proof} 
Immediate from the construction. 
\end{proof} 

\begin{proposition} 
The locally CAT(0) cubical complex $T_X$ is metrically complete 
if and only if $X$ has no infinite ascending chain of simplices.  
\end{proposition} 

\begin{proof} 
By Theorem~\ref{thm:cubecomplete}, 
a locally CAT(0) cubical complex is metrically complete if and 
only if it contains no infinite ascending chain of cubes.  It is 
immediate from the construction that this condition holds for 
$T_X$ if and only if $X$ contains no infinite ascending chain 
of simplices.  (Equivalently, again by Theorem~\ref{thm:cubecomplete}, 
$T_X$ is complete if and only if the all-right metric on $X$ is
complete.) 
\end{proof} 

\begin{proposition} 
There is a locally CAT(0) cubical complex $T'_X$, equipped with 
an involution $\tau'$ and a map $t'_X:T'_X\rightarrow X$ which 
has all of the properties listed in Theorem~\ref{thm:deltaversion}, 
and such that $T'_X$ is metrically complete for every $X$.  The 
map $t'_X:T'_X\rightarrow X$ is proper if and only if $X$ is 
finite-dimensional.  
\end{proposition} 

\begin{proof} 
If $X$ is $n$-dimensional, define $T'_X$ to be the iterated mapping
cylinder of the inclusion maps 
\[ T_{X^0} \rightarrow T_{X^1}\rightarrow \cdots \rightarrow
T_{X^n}.\]
From the definition of this space comes a map $p_X:T'_X\rightarrow
T_X$.  In general, define $T'_X$ to be the union of the subspaces
$T'_{X^n}$, and define a map $p_X:T'_X\rightarrow T_X$ as the direct 
limit of the maps $p_{X^n}$.  Now define $t'_X=t_X\circ p_X$.  The 
claimed properties follow readily from these definitions.  
\end{proof} 

\begin{remark} 
If $x_0$ is a vertex of $X$, then $x_0'=T_{x_0}$ is a vertex of
$T_X$ which is fixed by the $\tau$-action.  Thus $t_X$ induces 
maps of fundamental groups 
\[{t_X}_*:\pi_1(T_X,x_0')\rightarrow \pi_1(X,x_0),\quad 
{t_X}_*:\pi_1(T_X^\tau,x_0')\rightarrow \pi_1(X,x_0),\] 
which are easily seen to be surjective.  Note also that 
$\tau$ induces an automorphism $\tau_*$ of $\pi_1(T_X,x_0')$.  
Since $T_X^\tau$ includes into $T_X$ as a totally geodesic 
subcomplex, it follows that $\pi_1(T_X^\tau,x_0')$ is equal 
to the centralizer of $\tau_*$ in $\pi_1(T_X,x_0')$.  
\end{remark} 

\section{Arbitrary spaces} 
\label{sec:eig}

If $X$ is any topological space, then the first barycentric
subdivision of the space of singular simplicies, $\sing(X)$, is in the
category $\calc$ (the category of $\Delta$-complexes in which all the
edges of each simplex are distinct).  Composing this functor with our 
construction $T$ gives the following result.  

\begin{theorem} 
There is a functor $\hatT$ from topological spaces to the category
$\cald$ as defined in Theorem~\ref{thm:deltaversion}, an involution 
$\tau$ of $\hatT_X$ and a map $\hatt_X:\hatT_X\rightarrow X$ such that 

\begin{enumerate}
\item 
The map $\hatt_X$ induces an isomorphism on singular homology for any 
local coefficients on $X$.  

\item 
The involution $\tau$ on $\hatT_X$ has the property that $\hatt_X\circ
\tau= \hatt_X$ and the induced map $\hatT_X/\langle \tau\rangle 
\rightarrow X$ is a weak homotopy equivalence. 

\item 
The map $\hatt_X:\hatT_X^\tau
\rightarrow X$ induces an isomorphism on singular homology for any local
coefficients on $X$, where $\hatT_X^\tau$ denotes the fixed point set 
in $\hatT_X$ for the action of $\tau$.  

\item A continuous map $f:X\rightarrow Y$ gives rise to a cubical
  map $\hatT_f:\hatT_X\rightarrow \hatT_Y$.  

\item If $f$ is injective, then $\hatT_f$ embeds $\hatT_X$ as a totally
  geodesic subcomplex of $\hatT_Y$.  

\item If $f$ is locally injective, then $\hatT_f$ is a locally isometric
  map.  

\end{enumerate} 
\end{theorem}

\begin{remark} 
The theorem requires no connectivity assumptions on $X$ provided that 
singular homology with local coefficients is defined as
in~\cite[Ch.~VI.2]{whitehead}.  
\end{remark} 

\begin{proof} 
The only point that needs some verification is that a locally
injective map of topological spaces $f:X\rightarrow Y$ induces a
locally injective map of $\Delta$-complexes 
$\sing(f):\sing(X)\rightarrow \sing(Y)$.  
\end{proof}

\section{Generalized and equivariant homologies}
\label{sec:nin} 

Let $\calk$ denote any generalized homology theory satisfying 
the direct sum axiom (equivalently, the corresponding 
reduced theory satisfies Milnor's wedge axiom).  
A standard argument gives: 

\begin{proposition} 
For any $X\in \calc$ and any $\calk$, the map $t_X$ induces an isomorphism 
from $\calk_*(T_X)$ to $\calk_*(X)$.
\end{proposition} 

\begin{proof} 
Prove also the corresponding statement for relative homology.  
First consider the case of a pair $(X,Y)$, where $X$ is
$n$-dimensional and $Y$ is a subcomplex
of $X$ so that $X-Y$ consists of a single $n$-simplex.  In this case 
the Atiyah-Hirzebruch spectral sequence for computing the relative 
groups $\calk_*(T_X,T_Y)$ collapses, and by naturality of this 
spectral sequence one sees that the map $t_X:(T_X,T_Y)\rightarrow
(X,Y)$ induces an isomorphism $\calk_*(T_X,T_Y)\rightarrow
\calk_*(X,Y)$.  The case when $X=X^n$ is $n$-dimensional and 
$Y$ is the $(n-1)$-skeleton of $X$ follows (here we use the 
direct sum axiom).  

For finite-dimensional $X$, the absolute case follows from the 
relative case by induction.  The general case follows from this, 
since (again using the direct sum axiom) one knows that the
natural map from the direct limit of the $\calk_*(X^n)$ to 
$\calk_*(X)$ is an isomorphism, and similarly the natural map 
from the direct limit of the $\calk_*(T_{X^n})$ to $\calk_*(T_X)$ 
is an isomorphism.  The general relative case now follows.  
\end{proof} 

Now suppose that $G$ is a group acting on $X\in \calc$, and that 
$\calk^G_*$ is a $G$-homology theory in the sense of~\cite{lueck}.  
(Roughly speaking, this is a sequence of functors from pairs of 
$G$-CW-complexes to abelian groups which satisfies analogues of 
the Eilenberg-Steenrod axioms (except the dimension axiom) and 
a direct sum axiom.)  Then $G$ also acts on $T_X$, and a mild 
generalization of the argument given above gives: 

\begin{proposition} 
For any such $X$, $G$ and $\calk^G_*$, the map $t_X:T_X\rightarrow X$ 
is $G$-equivariant and induces an isomorphism
$\calk^G_*(T_X)\rightarrow \calk^G_*(X)$.  
\end{proposition}

Before continuing, we make some remarks concerning equivariant 
homology theories.  L\"uck has given two slightly different 
definitions of an equivariant homology theory $\calk^?_*$, 
one more restrictive than the other.  The results of this 
section are valid for the less restrictive definition, but 
we shall require the more restrictive definition in
Section~\ref{sec:ele}, and so we briefly describe both.  
The less restrictive definition is given in~\cite[sec.~1]{lueck}.  
The more restrictive definition is given
in~\cite[def.~4.2]{luecksur}.  The special case
of~\cite[def.~1.3]{bartechtlueck} (a definition of an `equivariant 
homology theory over a group $\Gamma$') in which the group 
$\Gamma$ is trivial is equivalent to~\cite[def.~4.2]{luecksur}.  

In each case we are given, for each group $G$ a $G$-homology 
theory $\calk^G_*$, together with induction maps that relate
these theories for different groups.  
If $\alpha:H\rightarrow G$ is 
a group homomorphism and $X$ is an $H$-CW-complex, there is a 
$G$-CW-complex $\ind_\alpha(X)$, defined as the quotient space:  
\[G\times X / (g,x)\sim (g\alpha(h),h^{-1}x) \qquad\hbox{for all $h\in H$}.\]
In the more restrictive definition, there is a map $\ind_\alpha:
\calk^H_*(X)\rightarrow \calk^G_*(X)$ for every $H$-CW-complex $X$ 
and every group homomorphism $\alpha$.  Moreover, the map
$\ind_\alpha$ is required to be an isomorphism whenever $\ker(\alpha)$ 
acts freely on $X$.  In the less restrictive definition, the map 
$\ind_\alpha$ is only required to be defined in the case when 
$\ker(\alpha)$ acts freely on $X$ (when it is still an isomorphism).

For the remainder of this section, we assume that $\calk^?_*$ is an 
equivariant homology theory in the sense of~\cite{lueck}.  
Before stating the main result of this section, we require one more 
definition.  For $G$ a discrete group, a $G$-CW-complex $E$ is said to 
be a model for $\ebar G$ if all stabilizer subgroups for $E$ are
finite and if for every finite subgroup $K\leq G$, the fixed point 
set $E^K$ is contractible.  It is easily shown that such an $E$ is 
unique up to equivariant homotopy equivalence.  

\begin{theorem} \label{thm:genequiv}
Let $G$ act on $X\in \calc$, so that the stabilizer of each 
simplex is finite.  Assume also that $X/G$ is connected.  
There is a group $\tilg$, a group homomorphism
$\alpha:\tilg\rightarrow G$ with $\ker(\alpha)$ torsion-free, 
and a $\tilg$-equivariant map $t:\ebar
\tilg\rightarrow X$ such that for any equivariant homology 
theory $\calk^?_*$, the map 
\[t_*:\calk^\tilg_*(\ebar \tilg)\rightarrow \calk^G_*(X)\] 
is an isomorphism.  In the case when $X$ is connected, the
homomorphism $\alpha$ is surjective.  
\end{theorem} 

\begin{proof} 
First, suppose that $X$ is not connected.  In this case pick $X_0$, a
component of $X$, and let $G_0$ be the setwise stabilizer of $X_0$.  
If $\beta$ denotes the inclusion homomorphism $\beta:G_0\rightarrow
G$, then $X$ is equivariantly homeomorphic to $\ind_\beta(X_0)$, 
and $\ind_\beta:\calk^{G_0}_*(X_0)\rightarrow \calk^G_*(X)$ is 
an isomorphism.  Hence it suffices to consider the case when $X=X_0$ 
is connected.  

In the connected case, consider $t_X:T_X\rightarrow X$.  By
naturality the group $G$ acts as a group of automorphisms of $T_X$.  
Let $E$ be the universal covering space of $T_X$, define 
$t:E\rightarrow X$ to be the composite of the projection map
$E\rightarrow T_X$ and $t_X:T_X\rightarrow X$, and let $\tilg$ be the 
group of all self-maps of $E$ that lift the action of $G$ on $T_X$.  
By construction, there is a surjective homomorphism 
$\alpha:\tilg\rightarrow G$ with $\ker(\alpha)$ equal to the group
of deck transformations of the covering $E\rightarrow T_X$.  By the 
previous proposition, the map $\calk^G_*(T_X)\rightarrow \calk^G_*(X)$
is an isomorphism.  Since the group of deck transformations acts
freely on $E$, it follows that $\ind_\alpha(E)\cong
E/\ker(\alpha)=T_X$ and that the map $\ind_\alpha:\calk^\tilg_*(E)\rightarrow 
\calk^G_*(X)$ is defined and is an isomorphism.  
But $E$ is a CAT(0) cubical complex, 
and $\tilg$ acts cellularly on $E$ with finite stabilizers.  It
follows from Theorem~\ref{ebarg} that $E$ is a model for 
$\ebar \tilg$.  Thus if we put $t=t_X\circ\ind_\alpha:E\rightarrow X$
we get the claimed result.  
\end{proof}

\section{Borel equivariant cohomology} 
\label{sec:ten}

Consider the Borel equivariant cohomology of a $G$-CW-complex $X$, 
defined as $H^*_G(X):= H^*(EG\times_G X)$.  Note that in the case 
when $X$ is contractible, this is just the cohomology $H^*(G)$ of 
the group $G$.  In general, by the simplicial 
approximation theorem, we may assume that $X$ and the $G$-action 
on $X$ are in $\calc$.  If all stabilizers in $X$ are finite and 
$X/G$ is connected, then with the notation of
Theorem~\ref{thm:genequiv}, we see that there is a chain of isomorphisms
$$t^*:H^*_G(X)\rightarrow H^*_\tilg(\ebar\tilg)\cong 
H^*(\tilg).$$
In particular, for any such $G$ and $X$, there is a group $\tilg$ 
for which the group cohomology ring $H^*(\tilg)$ is isomorphic to 
the Borel equivariant cohomology ring $H^*_G(X)$.  

Let $p$ be a prime, and let $\ff_p$ be the field of $p$ elements.  In
theorem~14.1 of \cite{quillen}, Quillen described, up to
F-isomorphism, the cohomology ring $H^*(\Gamma;\ff_p)$ of any discrete
group $\Gamma$ of finite virtual cohomological dimension over $\ff_p$.
In part~I of~\cite{quillen}, Quillen had already described, up to
F-isomorphism, the Borel equivariant cohomology ring
$H^*_G(X;\ff_p)$ for any finite $G$ and finite-dimensional
$G$-CW-complex $X$.  The proof of theorem~14.1, contained in sections
15~and~16 of~\cite{quillen}, involved exhibiting an isomorphism
between $H^*(\Gamma;\ff_p)$ and $H^*_G(X)$ for some finite $G$ and
suitable $X$.  In the case when $\Gamma$ has finite virtual
cohomological dimension over $\zz$, $G$ may be taken to be $\Gamma/N$
for any finite-index torsion-free normal subgroup $N$, and $X$ may be
taken to be $\ebar \Gamma/N$.  (Note that a group of finite vcd over $\zz$ 
necessarily has finite vcd over $\ff_p$, but the converse does not 
always hold.)  

One corollary of Theorem~\ref{thm:genequiv} is that this process can
be reversed: if we take as given Quillen's description of
$H^*(\Gamma;\ff_p)$ for all groups $\Gamma$ of finite vcd over the 
integers, then we can deduce Quillen's description of the Borel
equivariant cohomology ring $H^*_G(X;\ff_p)$ for all finite groups $G$
and all finite-dimensional $G$-CW-complexes $X$.  Moreover, we can 
deduce a theorem not stated in~\cite{quillen} by similar 
methods.  To justify this assertion, we need to define transport categories.  

Let $X$ be a $G$-CW-complex with all stabilizers finite.  
Following~\cite{quillen}, we define the transport category 
$\calt(G,X)$.  
The objects of $\calt(G,X)$ are pairs $(H,c)$, where $H$ is a finite
subgroup of $G$ and $c$ is a component of the $H$-fixed point set 
$X^H$.  The morphisms from $(H,c)$ to $(H',c')$ are the group
homomorphisms $H\rightarrow H'$ of the form $h\mapsto ghg^{-1}$ 
for some $g\in G$ such that $g.c$ is contained in~$c'$.  In the 
case when $X=\ebar G$, the transport category $\calt(G,\ebar G)$ 
is isomorphic to the Frobenius category $\Phi(G)$.  The objects of 
$\Phi(G)$ are the finite subgroups of $G$, with morphisms the group
homomorphisms that are induced by inner automorphisms of $G$. 
For $p$ a prime, let $\cala_p(G,X)$ denote the full subcategory 
of $\calt(G,X)$ with objects the pairs $(A,c)$ such that $A$ is 
abelian of exponent~$p$.  Quillen's description of $H^*_G(X;\ff_p)$, 
for $G$ finite, is as an inverse limit over $\cala_p(G,X)$.  

\begin{proposition} With notation as in the statement of
  Theorem~\ref{thm:genequiv}, the map $t:\ebar\tilg \rightarrow X$ 
induces an equivalence from the Frobenius category $\Phi(\tilg)$ to 
the transport category $\calt(G,X)$.  
\end{proposition} 

\begin{proof} Since the kernel of the map $\alpha:\tilg\rightarrow G$
  is torsion-free, any finite subgroup of $\tilg$ maps isomorphically 
to a finite subgroup of $G$.  Let $(H,c)$ be any object in
$\calt(G,X)$, 
let $T_c$ be the inverse image of $c$ in $T_X$, and let $c'$ be any 
lift of $T_c$ to $E$, the universal cover of $T_X$.  For any 
$x\in c'$, there is a subgroup $H'$ of $\tilg$ which fixes $x$ and maps
isomorphically to $H$.  If $c''$ is another lift of $c$, and $H''$ is 
defined similarly to $H'$, then there is an element $g\in \tilg$ such 
that $g.c'=c''$ and $gH'g^{-1}=H''$, and so $(H',c')$ is isomorphic to 
$(H'',c'')$ as objects of $\Phi(\tilg)$.  
\end{proof} 

\begin{corollary}
\label{cor:twentytwo}
For any virtually torsion-free group $G$ and any finite-dimen\-sion\-al 
$G$-CW-complex $X$ with finite stabilizer subgroups, the natural map 
$$H^*_G(X;\ff_p)\rightarrow \inverselim_{(A,c)\in \cala_p(G,X)}
H^*(A;\ff_p)$$ 
is a uniform F-isomorphism.  
\end{corollary} 

\begin{proof} 
The statement easily reduces to the case when $X/G$ is connected.  In
this case, by Theorem~\ref{thm:genequiv}, there exists a discrete
group $\tilg$, a homomorphism $\alpha:\tilg\rightarrow G$ and a
$\tilg$-equivariant map $t:\ebar\tilg\rightarrow X$ inducing an
isomorphism from $H^*(\tilg;\ff_p)$ to $H^*_G(X;\ff_p)$.  The map $t$
also induces an equivalence of categories from $\Phi(\tilg)$ to
$\calt(G,X)$.  Since $G$ is virtually torsion-free and $X$ is
finite-dimensional, it follows that $\tilg$ is of finite virtual
cohomological dimension.  The result now follows from Theorem~14.1
of~\cite{quillen}.
\end{proof}

\section{Assembly conjectures} 
\label{sec:ele}

A family $\calf$ of subgroups of a group $G$ is a non-empty collection
of subgroups which is closed under conjugation and taking subgroups,
i.e., for any $H\in \calf$, any $g\in G$ and any $K\leq H$,
$gHg^{-1}\in \calf$ and $K\in \calf$.  Examples include the family 
of all subgroups of $G$, the family of finite subgroups, and the 
family consisting of just the trivial subgroup.  

A model for $E(\calf,G)$, the classifying space for actions of $G$
with stabilizers in $\calf$, is a $G$-CW-complex $E$ such that each
cell stabilizer is in $\calf$ and such that for any $H\in \calf$, the
$H$-fixed point set $E^H$ is contractible.  It can be shown that 
for any $G$ and $\calf$ there exist models for $E(\calf,G)$, and 
that for fixed $G$ and $\calf$ any two such models are equivariantly 
homotopy equivalent.  In the case when $\calf$ is the family of all 
subgroups of $G$, a single point $\pt$ is a model for $E(\calf,G)$.  In 
the case when $\calf$ is just the trivial subgroup, a model for 
$E(\calf,G)$ is the same thing as a model for $EG$, the universal 
free $G$-CW-complex.  In the case when $\calf$ is the family of 
finite subgroups of $G$, a model for $E(\calf,G)$ is the same thing
as a model for $\ebar G$, as defined just above the statement of 
Theorem~\ref{thm:genequiv}.  If $G$ acts cellularly on a CAT(0) cubical
complex in such a way that all cell stabilizers are finite, then 
from Theorem~\ref{ebarg} it follows that $X$ is a model for $\ebar
G$.  

Let $\calk^?_*$ be an equivariant homology theory in the sense
of~\cite{lueck}, let $G$ be a group, and let $\calf$ be a family 
of subgroups of $G$.  The $(\calk^?_*,\calf,G)$ assembly conjecture
states that the map $E(\calf,G)\rightarrow \pt$ induces an isomorphism
from $\calk^G_*(E(\calf,G))$ to $\calk^G_*(\pt)$.  

For example, if $\calf$ is the family of all subgroups, then for any
$\calk^?_*$ and any group $G$, the assembly conjecture holds.  This 
example illustrates the principle that assembly conjectures are
stronger when they involve \emph{smaller} families of groups.  One way
to look at an assembly conjecture is as saying that $\calk^G_*(\pt)$ is 
determined by knowledge of $\calk^H_*(\pt)$ for groups $H\in \calf$, 
together with knowledge of how the $\calf$-subgroups of $G$ fit
together (embodied in the structure of $E(\calf,G)$).  

Another 
example of an assembly conjecture that holds for trivial reasons comes 
from Borel equivariant homology.  For any contractible $G$-CW-complex 
$E$, one has that $EG\times E$ is a free, contractible $G$-CW-complex.  
Hence the assembly conjecture for Borel equivariant homology is valid 
for any $G$ and any family $\calf$ of subgroups of $G$.

For the remainder of this section, let $\frakf$ denote the family 
of finite subgroups of a group, so that for any $G$ a model for 
$E(\frakf,G)$ is the same thing as a model for $\ebar G$.  Now 
suppose that $\calk^?_*$ is a generalized equivariant homology 
theory in the sense of~\cite{luecksur}
such that the $(\calk^?_*,\frakf,H)$ assembly conjecture 
holds whenever $H$ is a CAT(0) cubical group, i.e., whenever $H$ 
admits a cocompact cubical action on a CAT(0) cubical complex 
with stabilizers in $\frakf$.  Suppose further that $\calk^?_*$ is 
\emph{continuous} in the sense that for any group $G$, and any 
expression for $G$ as a directed union $G=\cup_i G_i$, one 
has that the natural map 
\[\lim_i\calk^{G_i}_*(\pt)\rightarrow \calk^G_*(\pt)\] 
is an isomorphism~\cite{bartechtlueck}.

\begin{theorem} \label{thm:assembly}
  For a group $G$ and a model $X$ for $\ebar G$, define
  a group $\tilg$ and a homomorphism $\alpha:\tilg\rightarrow G$ as in
  the statement of Theorem~\ref{thm:genequiv}.  For a generalized
  equivariant homology theory $\calk^?_*$ satisfying the hypotheses
  listed in the previous paragraph, the following statements are
  equivalent:  
\begin{itemize} 

\item The $(\calk^?_*,\frakf,G)$ assembly conjecture holds

\item The map $\ind_\alpha: \calk^\tilg_*(\pt)\rightarrow 
\calk^G_*(\pt)$ is an isomorphism

\end{itemize} 
\end{theorem}

\begin{proof} 
First we claim that the $(\calk^?_*,\frakf,\tilg)$ assembly conjecture
holds.  Let $\{X_i:i\in \cali\}$ denote the collection of
$G$-subcomplexes of $X$ that contain only finitely many $G$-orbits of
simplices and have the property that $X_i/G$ is connected.  Then $X$ 
is the directed union of the subspaces $X_i$.  Applying the
construction of Theorem~\ref{thm:genequiv} to these spaces $X_i$ gives
an expression for $\tilg$ as a directed union of subgroups $\tilg_i$.  
The universal cover of $T_{X_i}$ is a CAT(0) cubical complex admitting
a cocompact action of $\tilg_i$ with stabilizers in $\frakf$.  Hence 
the $(\calk^?_*,\frakf,\tilg_i)$ assembly conjecture holds for each
$i\in \cali$.  Since $\calk^?_*$ is assumed to be continuous, the 
assembly map for $\tilg$ is equal to the following composite: 
\[\calk^\tilg_*(\ebar\tilg)\cong 
\lim_i \calk^{\tilg_i}_*(\ebar\tilg_i) \rightarrow 
\lim_i \calk^{\tilg_i}_*(\pt)\cong 
\calk^\tilg_*(\pt).\] 
It follows that the $(\calk^?_*,\frakf,\tilg)$ assembly conjecture holds 
as claimed.  

The homomorphism $\alpha:\tilg\rightarrow G$ induces an equivariant map 
$\ebar\tilg\rightarrow \ebar G$ and hence a commutative square as
below:
\[
\begin{array}{ccc}
\calk^\tilg_*(\ebar\tilg)&\rightarrow& \calk^G_*(\ebar G)\\
\downarrow&&\downarrow \\
\calk^\tilg_*(\pt)&\rightarrow& \calk^G_*(\pt),\\
\end{array}
\] 
in which the horizontal maps are both called $\ind_\alpha$, and 
come from the identifications $\ind_\alpha(\ebar\tilg)=\ebar G$ 
and $\ind_\alpha(*)=*$.  
By Theorem~\ref{thm:genequiv}, the map down the left side of the
diagram is an isomorphism, and by the argument given in the first
paragraph the top horizontal map is an isomorphism.  It follows that 
the lower horizontal map is an isomorphism if and only if the map down
the right side is an isomorphism.  
\end{proof}

There are a number of assembly conjectures to which
Theorem~\ref{thm:assembly} applies.  In each case discussed below, the
smoothness of the relevant homology theory is established
in~\cite{bartechtlueck}.  

\begin{itemize}

\item
The Bost conjecture is the assembly conjecture for a homology theory 
in which $\calk^G_*(\pt)$ is isomorphic to the topological $K$-theory 
of $\ell^1(G)$.  In~\cite{nibloreeves} it is shown that CAT(0) cubical 
groups have the Haagerup property.  In~\cite{lafforgue} the Bost conjecture is 
established for any group having the Haagerup property.  Thus
Theorem~\ref{thm:assembly} applies to this case.  

\item
Theorem~\ref{thm:assembly} does not apply directly to the Baum-Connes
conjecture, an assembly conjecture for which $\calk^G_*(\pt)$ is
isomorphic to the topological $K$-theory of the reduced group
$C^*$-algebra $C_\red^*(G)$.  The Baum-Connes conjecture holds for
groups having the Haagerup property~\cite{HigKas}, and the homology
theory is smooth~\cite{bartechtlueck}.  The difficulty is that the
homomorphism $\alpha:\tilg\rightarrow G$ does not induce a
homomorphism from $C_\red^*(\tilg)$ to $C_\red^*(G)$, and so 
there is no known way to define the map $\ind_\alpha$ appearing 
in the statement of Theorem~\ref{thm:assembly}.  
Instead of Theorem~\ref{thm:assembly}
we can make the following slightly weaker statement:  with notation as 
in the statement of Theorem~\ref{thm:assembly} and with $\calk^?_*$
equal to the generalized equivariant homology theory appearing in the 
Baum-Connes conjecture, we have that the Baum-Connes conjecture holds 
for $\tilg$ and $\ind_\alpha: \calk^\tilg_*(\ebar \tilg)\rightarrow
\calk^G_*(\ebar G)$ is an isomorphism.  
%
In the special case when $G$ is torsion-free (and with a different 
construction of a group $\tilg$) a similar statement appears in 
\cite[Introduction]{block}~and~\cite[part~I, theorem~5.19]{mislin}.  

\item
Let $R$ be a ring equipped with an involution.  After inverting the
prime~2, the Farrell-Jones conjecture in algebraic $L$-theory with 
coefficients in $R$ is equivalent to to the $(\calk^?_*,\frakf,G)$
assembly conjecture, where $\calk^H_*(\pt)$ is isomorphic to
$\zz[1/2]\otimes {\mathbf L}_*^{\langle-\infty\rangle}(RH)$, the 
algebraic $L$-groups of $RH$ localized away from the prime~2.  The 
Farrell-Jones conjecture in algebraic $L$-theory is proved for 
CAT(0) groups in~\cite{barlue}, and hence Theorem~\ref{thm:assembly}
applies directly for this choice of $\calk^?_*$.  

\item
Let $R$ be a regular ring, and suppose that $G$ is a group with the
property that the order of every finite subgroup of $G$ is a unit in
$R$.  Under these circumstances the Farrell-Jones conjecture for the
algebraic $K$-theory of $RG$ is equivalent to the
$(\calk^?_*,\frakf,G)$ assembly conjecture where $\calk^?_*$ is a
generalized equivariant homology theory for which $\calk^G_*(\pt)$ is
isomorphic to the algebraic $K$-theory of $RG$.  (Without these extra
assumptions on $G$ and~$R$, the Farrell-Jones conjecture would require
the larger family of virutally cyclic subgroups rather than the family
$\frakf$ of finite subgroups.)  The Farrell-Jones conjecture for
CAT(0) groups is proved in~\cite{wegner} (see
also~\cite{barlue}).  Every finite subgroup of the group $\tilg$ in
Theorem~\ref{thm:assembly} is isomorphic to a subgroup of $G$.  Thus
provided that $R$ satisfies the hypotheses listed at the start of this
paragraph, Theorem~\ref{thm:assembly} applies in this case.  Thus 
the Farrell-Jones assembly map $\calk^G_*(\ebar G)\rightarrow 
\calk^G_*(\pt)$ is an isomorphism if and only if $\alpha$ induces 
an isomorphism from $K_*(R\tilg)$ to $K_*(RG)$.  

\end{itemize}

\section{The torsion subgroup of a CAT(0) group} 
\label{sec:twe}

In this section, we let $\bfc$ denote the class of groups that act
cocompactly by cellular isometries on a CAT(0) cubical complex, and
that contain a torsion-free subgroup of index at most two.  A simple way
to describe a group $G$ in $\bfc$ is by giving a finite locally
CAT(0) cubical complex $T$ together with a cellular involution $\tau$
on $T$: the group $G$ thus described is the group of all lifts to the
universal cover of $T$ of elements of the group generated by $\tau$,
and the fundamental group of $T$ is a torsion-free subgroup of $G$ of
index at most two.  Note that we allow the case when $\tau$ is the
identity map on $T$, in which case $G$ is equal to the fundamental
group of $T$.  


Note that it is easy 
to compute the number of conjugacy classes of elements of order two 
in $G$: if $\tau$ is not the identity map then 
there is a bijection between these conjugacy classes and the 
components of the fixed point set $T^\tau$.  Not every problem
involving the torsion in $G$ can be solved so easily however, as 
the following theorem shows.  

\begin{theorem} 
There is no algorithm to determine for all $G\in \bfc$ whether $G$ 
is generated by torsion.  
\end{theorem} 

\begin{proof} 
It is known that there is no algorithm to determine whether a given
finite presentation presents the trivial group.  In particular, we may
construct a sequence $X_n$ of finite connected simplicial complexes
such that no algorithm can compute the function $f:\nn\rightarrow
\{0,1\}$ defined by $f(n)=1$ if and only if $\pi_1(X_n)$ is trivial.  

Apply Theorem~A to the sequence $X_n$, and consider
the corresponding sequence $(T_{X_n},\tau)$ of finite locally CAT(0)
cubical complexes equipped with a cellular involution, which describes
a sequence $G_n$ of groups in $\bfc$.  The group $G_n$ is generated
by torsion if and only if $\pi_1(X_n)$ is trivial, and so there can be
no algorithm to decide whether or not each $G_n$ is generated by
torsion.
\end{proof} 

\section{Closing remarks} 
\label{sec:thi}

It is natural to ask whether there is a version of
Theorem~A in which `locally CAT(0)' is replaced by
the stronger condition `locally CAT($-1$)'.  The author posed this
question in~\cite[Q~1.23]{besques}.  As a first step in answering this
question which is of interest in its own right: can one 
construct high-dimensional locally CAT($-1$) spaces that are 
acyclic but not contractible?  

In~\cite[theorem~20.2]{janswi}, Januszkiewicz and \'Swi\c{a}tkowski
show that for any finite simplicial complex $X$, there exists a
CAT($-1$) simplicial complex $S_X$ and a group $H_X$ acting properly
cocompactly and simplicially on $S_X$, such that the quotient
$S_X/H_X$ is homotopy equivalent to $X$.  This result can be viewed as
a CAT($-1$) but non-natural version of property~\ref{condb} of 
Theorem~A for finite complexes.  Corollaries 20.4~and~20.5
of~\cite{janswi} answer a question posed by the author
in~\cite[Q~1.24]{besques}.  

Two questions that were raised in earlier versions of this paper have
been answered.  Pol\'ak and Wise have shown that the group presented 
in Proposition~\ref{prop:acycone} is residually
finite~\cite{polakwise}.  Raeyong Kim has shown that every finite 
connected $X$ has the same homology as some locally CAT(0) cubical complex 
$V_X$ whose fundamental group is a duality group~\cite{raeyong}; this
is a CAT(0) version of a theorem of Hausmann~\cite{Hausmann}.  

Finally, a comment on hyperbolization as in 
\cite[sec.~3.4]{gromov}~and~\cite{charneydavis}.  As in our
Kan-Thurston construction, one constructs for a simplicial complex
$X$, a locally CAT(0) cubical complex $W_X$ equipped with a map
$w_X:W_X\rightarrow X$.  The extra property enjoyed by $W_X$ is that
its vertex links are similar to those in $X$, so that if $X$ is a 
PL $n$-manifold then so is $W_X$.  The map on homology induced by 
$w_X$ can be taken to be a split surjection.  Consideration of the 
case when $X$ is a 2-sphere shows that there cannot be a `Kan-Thurston
hyperbolization' in general.  For $n>1$ it seems to be unknown whether 
there exists a aspherical homology $2n$-sphere.

\appendix
\section{Metrizing complexes} 
\label{sec:appone}

We collect some facts about simplicial and cubical complexes,
together with detailed references or proofs.  Since we intend 
to metrize our complexes, we refer to~\cite[I.7.40]{brihae}, 
parts (2)~and~(3) respectively for the definitions of a simplicial 
complex and a cubical complex.  

Throughout this section, each $n$-cube is metrized as a standard
Euclidean cube of side length 1, and each $n$-simplex is metrized as 
the subspace of the Riemannian sphere of radius 1 in $\rr^{n+1}$ 
defined by intersecting the sphere with $(\rr_{\geq 0})^{n+1}$.  
This metric on the $n$-simplex will be called the `all right metric'.  
Let $C$ denote either a simplicial complex or a
cubical complex, and refer to the simplices or cubes of $C$ as its
cells.  As in~\cite[I.5.9]{brihae}, the metric on the cells induces a
pseudo-metric on any connected simplicial or cubical complex, in which
the distance between points $x$ and $y$ is the infimum of sums of the
form $\sum_{i=1}^n d(x_{i-1},x_i)$, where $x_0=x$, $x_n=y$, and for
each $i$, there is a single cell containing $x_i$ and $x_{i-1}$, and
the distance is measured in the standard metric on that cell.
(A sequence of points $x_0,\ldots,x_n$ with these properties is called
a `string' or an `$n$-string'.)  
By~\cite[I.7.10]{brihae}, this pseudo-metric turns out to be a metric.
Moreover, this metric makes $C$ into a length space as defined
in~\cite[I.3.1]{brihae}.  Another treatment of this material is 
contained in~\cite[section~3]{moussong}.  
In the case when $C$ is finite-dimensional, 
it is shown in~\cite[I.7.13]{brihae} that this metric is complete.  
We shall characterize the complexes $C$ for  which the metric is
complete, but first we need some lemmas.  

\begin{proposition} \label{shortest} 
Let $\sigma$ be a cell of $C$ and $x$ a point of $C$, such that 
$d=d(\sigma,x)$ satisfies either $d<1$ if $C$ is cubical or 
$d<\pi/2$ if $C$ is simplicial.  Then there is straight line 
segment from $\sigma$ to $x$ of length $d$, which is contained 
in the smallest closed cell of $C$ containing $x$.  This line 
segment is unique.  
\end{proposition} 

\begin{proof} 
Pick an $n$ and an $n$-string $x_0,\ldots,x_n$ of length less than 
$1$ (resp.\ $\pi/2$) with $x_0\in \sigma$ and $x_n=x$, and suppose 
that $n\geq 2$.  Let $\sigma_i$ be the minimal closed cell containing 
the line segment $[x_{i-1},x_i]$.  We may assume that no $\sigma_i$ is
a face of $\sigma$.  If the dimensions of these cells 
satisfy $\dim(\sigma_2)>\dim(\sigma_1)$, then $\sigma_1$ is a face of 
$\sigma_2$ and the $(n-1)$-string $x_0,x_2,\ldots,x_n$ is shorter.  
If $\sigma_1=\sigma_2$, then once again the $(n-1)$-string 
$x_0,x_2,\ldots,x_n$ is either shorter or of the same length 
(the latter possibility 
occurs if $x_1$ is on the line segment $[x_0,x_2]$).  In the remaining 
case, the minimal closed cell $\tau$ containing $x_1$ is a proper face 
of $\sigma_1$.  Since the segment $[x_0,x_1]$ has length less than 
$1$ (resp.\ $\pi/2$), it follows that $\tau\cap \sigma$ is non-empty.  
Let $x'_0$ be the base of the perpendicular from $x_1$ to 
$\tau\cap\sigma$.  The 
$n$-string $x'_0,x_1,\ldots,x_n$ is a path from $\sigma$ to $x$ and is
shorter than the original string.  

This process one can reduce any $n$-string for $n\geq 2$ to the  
single line segment $[y,x]$, where $\tau$ is the minimal closed cell 
containing $x$ and $y$ is the base of the perpendicular in $\tau$ from 
$x$ to $\tau\cap\sigma$.  It follows that this segment is the unique 
shortest path from $\sigma$ to $x$, and that it has length $d$.  
\end{proof} 

\begin{corollary} 
Suppose that $\lambda<1$ in the cubical case or $\lambda<\pi/2$ in the 
simplicial case, and for $\sigma$ a cell of $C$, let
$O=O^\lambda_C(\sigma)$ be the open $\lambda$-ball in $C$ around $\sigma$.  
For any cell $\tau$ in $C$, one has that 
$$O^\lambda_C(\sigma)\cap\tau= O^\lambda_\tau(\sigma\cap\tau).$$ 
In particular $O^\lambda_C(\sigma)\cap\tau$ is empty if
$\sigma\cap\tau$ is empty.  
\end{corollary} 

\begin{corollary} \label{ftau} 
Let $\lambda=2/3$ and for $\tau$ a cube in a
cubical complex $C$ define $F_\tau$ by the equation 
$$F_\tau=C - \bigcup_{\sigma\cap\tau=\emptyset}
O^\lambda_C(\sigma).$$
Also define $F_\emptyset=\emptyset$.  
Then $F_\tau$ is closed and contains the closed ball in $C$ of radius 
$1/3$ about $\tau$.  If $\tau_i$ for $i\in I$ are cells of $C$,
with intersection $\mu$, then $\bigcap_{i\in I}F_{\tau_i}= F_\mu$.  
\end{corollary} 

\begin{proof} 
As the complement of a union of open sets, $F_\tau$ is closed.  
The previous corollary allows one to find the intersection of 
$F_\tau$ with each closed cell of $C$.  This gives a combinatorial 
description of $F_\tau$.  Let $C'$ be the subdivision of $C$ in 
which each $n$-cube is divided in to $3^n$ congruent pieces.  
Then $F_\tau$ is equal to the union of the (closed) cubes in $C'$ 
that intersect $\tau$.  Equivalently, $F_\tau$ is equal to the 
union of the (closed) cubes in $C'$ that contain the barycentre of
some cube of $C$.  The other claimed properties follow.  
\end{proof}

\begin{proposition} \label{nonemp}
Suppose that there is a point $x$ in a cubical complex $C$ which 
is at distance at most $1/3$ from a collection $\tau_i$ for $i\in I$ 
of cells of $C$.  Then $\mu=\bigcap_{i\in I} \tau_i$ is non-empty.  
\end{proposition} 

\begin{proof} 
The point $x$ is in each $F_{\tau_i}$, and hence $\bigcap_{i\in I} 
F_{\tau_i}= F_\mu$ is non-empty.  
\end{proof} 

\begin{remark} 
There is no analogue of this proposition for all right simplicial
complexes.  Let $C$ be the boundary of an $n$-simplex.  The 
top dimensional faces of $C$ do not intersect, but their barycentres
are at distance $\arccos(1-1/n)$ from each other.  
\end{remark} 

\begin{theorem} \label{thm:cubecomplete}
The metric on $C$ is complete if and only if $C$ does not contain an 
infinite ascending sequence of cells.  
\end{theorem} 

\begin{proof} 
Let $\sigma_0\subseteq\sigma_1\subseteq\cdots$ be an ascending
sequence of cells, where $\sigma_i$ has dimension $i$.  Let
$x_0=\sigma_0$, and for $i>0$ pick $x_i\in \sigma_i$ to be a point at
distance $1/3^i$ from the boundary of $\sigma_i$ whose nearest point 
in the boundary is $x_{i-1}$.  This is a Cauchy sequence with no
convergent subsequence.  

For the converse, first consider the case when $C$ is cubical.
Suppose that $C$ contains no infinite ascending sequence of cells.  
Let $(x_n)$ be a Cauchy sequence in $C$, and let $\sigma_n$ denote the
minimal closed cube containing $x_n$.  By passing to a subsequence, 
we may assume that each $d(x_m,x_n)<1/3$, and so by
Proposition~\ref{nonemp}, $\theta=\bigcap_n\sigma_n$ is non-empty.  
Suppose that $\theta'$ is a cell properly containing $\theta$ 
which is contained in $\sigma_n$ for infinitely many $n$.  By passing
to this subsequence, we may replace $\theta$ by $\theta'$.  Since
there is no infinite ascending chain of cells, we may assume that 
$\theta$ has the property that for each $\theta'>\theta$, there are 
only finitely many $n$ for which $\theta'\geq \sigma_n$.  Now let 
$y_n$ be the closest point of $\theta$ to $x_n$.  The map $x_n\mapsto
y_n$ is distance-decreasing, so $(y_n)$ is Cauchy.  (The reader who 
does not want to check that this map is distance-decreasing may 
instead pass to a convergent subsequence.)  Let $y$ be the 
limit of the sequence $(y_n)$.  

It remains to show that $(x_n)$ converges to $y$.  But if
$d(x_m,x_n)<1$, then certainly $d(\sigma_m,x_n)<1$, and so by 
Proposition~\ref{shortest}, the shortest path from $x_n$ to 
$\sigma_m$ is the straight line inside $\sigma_n$ from $x_n$ to 
$y$.  Thus $d(y,x_n)=d(\sigma_m,x_n)\leq d(x_m,x_n)$.  This completes 
the proof in the cubical case.  

For the simplicial case, given an all right simplicial complex $C$
with vertex set $V$, let $v_0$ be the origin in $W=\bigoplus_V\rr$,
equipped with the standard Euclidean norm, and identify the vertex
$v\in C$ with the unit vector in the positive $v$-direction in $W$.
For each $(n-1)$-simplex $\sigma=(v_1,\ldots,v_n)$ of $C$, let
$c(\sigma)$ be the $n$-cube in $W$ with vertex set the points
$\sum_{i=1}^n \epsilon_iv_i$, for $\epsilon_i\in\{0,1\}$.  Define a
cubical complex $D$ as the union of all of these (closed) cubes.  Now $D$
is a cubical complex in which the link of $v_0$ is equal to $C$.  In
particular, $D$ has no infinite ascending sequence of cubes if and
only if $C$ has no infinite ascending sequence of simplices.
Futhermore, there is a natural embedding of $C$ in $D$, which embeds
$\sigma$ as the intersection of $c(\sigma)$ and the unit sphere in $W$
based at $v_0$.  This embedding does not preserve distances: if $x,y$
are points inside a single simplex $\sigma$ of $C$ and
$d_\sigma(x,y)=\theta$, then $d_{c(\sigma)}(x,y)=2\sin(\theta/2)$.
The convexity properties of the sine function imply that for any
points $x,y\in C$ with $d_C(x,y)=\theta\leq \pi/4$, one has that
$2\sin(d_C(x,y)/2)\leq d_D(x,y)\leq d_C(x,y)$.  It follows that a
sequence $(x_n)$ in $C$ is Cauchy in $C$ if and only if it is Cauchy
in $D$.  If $x\in D$ is the limit in $D$ of the Cauchy sequence
$(x_n)$ in $C$, then by continuity of the metric, $d_D(v_0,x)=1$ and
so $x\in C$.
\end{proof} 

\begin{theorem} 
The identity map from $C$ with the CW-topology to $C$ with the metric
topology is a homotopy equivalence.  
\end{theorem} 

\begin{proof} 
This is due to Dowker~\cite{dow}.  
In the cubical case, there is an easy combinatorial description of a 
homotopy inverse to the identity map.  Let $C'$ denote 
the subdivision of $C$ in which each $n$-cube is cut into $3^n$
cubes.  There is a function $f$ from the set of cubes of $C'$ to 
the set of cubes of $C$, defined by $f(\sigma)=\tau$ if the closed 
cube $\sigma$ contains the barycentre of $\tau$.  For each $\tau\in C$, 
the set $f^{-1}(\tau)$ contains a unique cube $m(\tau)$ of minimal
dimension.  The cube $m(\tau)$ is the `middle cube' of $\tau$, and 
its dimension is equal to the dimension of $\tau$.  Every cube 
$\sigma\in f^{-1}(\tau)$ has $m(\tau)$ as a face.  Define a map 
$\phi:C'\rightarrow C$ as follows.  If $\sigma=m(\tau)$ for some 
$\tau$, then define $\phi:\sigma\rightarrow \tau$ by expanding all 
distances from the common barycentre of $\sigma$ and $\tau$ by a 
factor of 3.  Otherwise, take $\phi$ to be the composition of the 
orthogonal projection from $\sigma$ to $m(f(\sigma))$ with the map 
already defined from $m(f(\sigma))$ to $f(\sigma)$.

This map $\phi$ is cellular as a map from $C$ to $C$, and multiplies 
distances by a factor of at most 3.  It follows that in either
topology there is a homotopy 
from $\phi$ to the identity map which is linear on each cube.  It 
remains to check that $\phi$ is continuous from the metric topology to 
the CW-topology.  To see this, let $x$ be a point of $C$, and let 
$\sigma$ be the minimal cube 
of $C$ that contains $\phi(x)$.  If $O$ is any open set in $C$ containing 
$\phi(x)$, pick $\epsilon>0$ such that the intersection of the interior of 
$\sigma$ and $O$ contains the open $\epsilon$-ball around $\phi(x)$.  Then 
$\phi^{-1}(O)$ contains the open $\epsilon/3$-ball in $C$ around $x$.
This shows that $\phi$ is continuous.  
\end{proof} 

\begin{remark} 
The set $F_\tau$ defined in Corollary~\ref{ftau} is equal to
$\phi^{-1}(\tau)$.  This fact can be used to give a different proof 
of Corollary~\ref{ftau}. 
\end{remark} 


\section{The CAT($\kappa$) conditions}  
\label{sec:apptwo}

\begin{definition} 
Define a combinatorially CAT(0) cubical complex to be
a simply connected 
cubical complex in which all vertex links are flag.  (It will 
be shown in Theorem~\ref{gromovthm} that a cubical complex is CAT(0) 
if and only if it is combinatorially CAT(0).)  
\end{definition}

\begin{definition} 
A full subcomplex of a simplicial complex $C$ is a subcomplex $D$ such 
that whenever $\sigma\in C$ and each vertex $v$ of $\sigma$ is in
$D$, then $\sigma$ is in $D$.  
A combinatorially convex subcomplex of a cubical complex $C$ is a
connected subcomplex $D$ such that for each vertex $v$ of $D$, the
link in $D$, $\Lk_D(v)$ is a full subcomplex of $\Lk_C(v)$.  
\end{definition} 

\begin{proposition} 
Every finite subcomplex of a simplicial complex $C$ is contained
in some finite full subcomplex $D$.  If $C$ is connected, then $D$ 
may be taken to be connected.   
\end{proposition} 

\begin{proof} 
Let $C'$ be a finite subcomplex of $C$, and let $D$ be the full 
subcomplex of $C$ with the same vertex set as $C'$.  This is a 
finite full subcomplex since its vertex set is finite.  If $C'$ 
is connected, then so is $D$.  In the case when $C$ is connected, 
any finite subcomplex $C''$ is contained in a finite connected 
subcomplex $C'$ (e.g., the union of $C''$ and a choice of edge 
paths joining the components of $C''$), and so the extra claim 
in this case follows.  
\end{proof} 

\begin{theorem} \label{sagthm}
Every finite subcomplex of a combinatorially CAT(0) cubical complex is
contained in a finite combinatorially convex subcomplex.  
\end{theorem} 

\begin{proof} 
This uses Sageev's theory of hyperplanes~\cite{sag}.  A hyperplane in 
a combinatorially 
CAT(0) cubical complex $C$ is an equivalence class of directed edges 
under the transitive closure of the relation `there exists a square in 
which the two directed edges are parallel'.  The corresponding open 
hyperplane neighbourhood is the union of all open cubes whose closures
contain an
edge of the hyperplane.  If $H$ is an open hyperplane neighbourhood, 
then $C-H$ is a cubical complex with the same vertex set as $C$.  If 
$v$ is a vertex of $C$, then $\Lk_{C-H}(v)$ is the full subcomplex of 
$\Lk_C(v)$ with vertex set the (undirected) edges containing $v$ that 
are not contained in $H$.  We define the opposite hyperplane to $H$ to 
be the hyperplane which contains the reverse of every oriented edge in
$H$.  We quote four results which can easily be deduced from Sageev's 
paper~\cite{sag}.  For any square $\sigma\in C$, a hyperplane
contains at most two of the eight directed edges of
$\sigma$~\cite[lemma~4.8 and corollary~4.9]{sag}.   
For each open hyperplane neighbourhood $H$, $C-H$ has two
components~\cite[theorem~4.10]{sag} (this implies that no hyperplane 
is equal to its own opposite).  Call these two components the half-spaces
defined by $H$.   If the shortest 1-skeleton path between vertices $v$
and $w$ has length $n$, then there are exactly $n$ opposite pairs of 
hyperplanes in $C$ 
that separate $v$ and $w$ (the lower bound comes
from~\cite[theorem~4.13]{sag}, and the upper bound
from~\cite[theorem~4.6]{sag}).  A vertex $v$ of $C$ is
uniquely determined by the collection of half-spaces in which it
lies (follows from \cite[theorems 4.6~and~4.10]{sag}, since if 
$e$ is an edge on any 1-skeleton geodesic from $v$ to $w$ then 
the hyperplane containing $e$ separates $v$~and~$w$).  

Given a finite subcomplex $C'$ of $C$, let $V'$ be the vertex set of 
$C'$, and let $D$ be the intersection of all half-spaces of $C$ that 
contain $V'$.  As an intersection of combinatorially convex
subcomplexes, $D$ is combinatorially convex.  Suppose that $V'$
contains $n$ elements and that each pair $v,w\in C$ can be joined by a
1-skeleton path of length at most $N$ in $C$.  Then there are at most 
$M=N{\binom{n}{2}}$ hyperplanes of $C$ that separate points of $V'$.  
Since vertices of $C$ are uniquely determined by the half-spaces that
contain them, it follows that $D$ contains at most $2^M$ vertices.  
\end{proof} 


Before proving the main theorems of this section, we recall a
definition from~\cite[I.7.8]{brihae}.  Suppose that $C$ is a 
simplicial or cubical complex in which cells have been given 
compatible metrics (i.e., so that the inclusion of each face 
of each cell is an isometry).  Under these circumstances, 
for $x$ a point of $C$ and $\sigma$ any cell of $C$ containing 
$x$, define $\epsilon(x,\sigma)$ to be the infimum of the 
distances from $x$ to points in closed faces of $C$ that do not 
contain $x$.  Now define $\epsilon_C(x)$ to be the infimum of 
the numbers $\epsilon(x,\sigma)$ as $\sigma$ ranges over all 
closed cells of $C$ that contain $x$.  Informally, we may think 
of $\epsilon(x)$ as the infimum of the distances from $x$ to 
points $y$ in closed cells that do not contain $x$.  

\begin{proposition}\label{epsilon} 
Let $C$ be either an all-right simplicial complex or a cubical 
complex.  Then for any point $x$ of $C$, we have that 
$\epsilon_C(x)=\epsilon_\sigma(x)$, where $\sigma$ is the unique
closed cell of $C$ such that $x$ is in the interior of $\sigma$. 
In particular, $\epsilon_C(x)>0$ for all $x$.  
\end{proposition} 

\begin{proof} 
An easy exercise.
\end{proof}

\begin{theorem}\label{flag}  
Let $C$ be an all-right simplicial complex.  Then $C$ is CAT(1) if and
only if $C$ is flag.  
\end{theorem} 

\begin{proof} 
The case when $C$ is finite dimensional is proved
in~\cite[II.5.18]{brihae}.  The argument given there to show that
CAT(1) implies flag generalizes easily to the infinite-dimensional
case.  If $C$ is not flag, take a `missing simplex' of minimal 
dimension, i.e., a subset $v_0,\ldots,v_n$ of the vertex set of $C$, 
where $n$ is at least 2 and is  minimal subject to the property that the set 
$\{v_0\ldots,v_n\}$ does not span a simplex of $C$ but any proper
subset does.  If $n=2$, then the union of the three edges joining 
the vertices $v_0,v_1,v_2$ is a closed local geodesic of length less than 
$2\pi$, showing that $C$ cannot be CAT(1).  For $n>2$, argue by
induction on~$n$.  The link $\Lk_C(v_n)$ contains the `missing simplex' 
$v_0,\ldots,v_{n-1}$ of dimension $n-1$, and so $\Lk_C(v_n)$ is not
CAT(1).  It follows from~\cite[II.5.2]{brihae} that $C$ cannot be
CAT(1).  (In the infinite dimensional case, the hypotheses
of~\cite[II.5.2]{brihae} are not satisfied, but
Proposition~\ref{epsilon} shows that the alternative hypotheses listed
in~\cite[II.5.3]{brihae} do hold.)  

The general case of `flag implies CAT(1)' will be proved below.  
\end{proof}

\begin{theorem}\label{localisom} 
Let $i:D\rightarrow C$ be the inclusion of a full subcomplex in a flag 
simplicial complex.  Then $i$ preserves distances less than $\pi$, in 
the sense that for $x,y\in D$, one has that $d_D(x,y)<\pi$ if and only 
if $d_C(i(x),i(y))<\pi$ and if this holds, then
$d_D(x,y)=d_C(i(x),i(y))$.  
\end{theorem} 

\begin{proof} The proof will rely on Theorem~\ref{flag}, so for now 
it will only cover the case when $C$ is finite-dimensional.  This 
case will then be used to prove the general case of
Theorem~\ref{flag}, and then the proof given below will cover all 
cases.  

Write $C*_DC$ for the double of $C$ along $D$, i.e., the quotient
space obtained from two copies of $C$ by identifying the two copies of
$D$.  Each of the two inclusions $C\rightarrow C*_DC$ is a map that 
does not increase distance (since a string in $C$ maps to a string in 
$C*_DC$ of the same length).  The composite map 
$$C\rightarrow C*_DC\rightarrow C*_CC = C$$ 
is the identity, and hence each of the inclusions $C\rightarrow C*_DC$
is an isometric embedding.  

It is easily seen that $C*_DC$ is flag.  By the cases of
Theorem~\ref{flag} that are already known, if $C$ is
finite-dimensional, then $C$, $D$ and $C*_DC$ are all CAT(1).  
The map $i:D\rightarrow C$ does not increase distance.  The self-map
of $C*_DC$ that swaps the two copies of $C$ is an isomorphism and so 
is an isometry.  Suppose 
that $x,y\in D$ are such that $d_C(i(x),i(y))<\pi$.  Then
by~\cite[II.1.4(1)]{brihae}, there is a unique geodesic in $C*_DC$ 
from $x$ to $y$.  If the geodesic did not lie entirely inside $D$, 
there would be at least two such geodesics.  Hence the geodesic 
does lie inside $D$, and so $d_D(x,y)=d_C(x,y)$.  
\end{proof} 

\begin{proof} (Theorem~\ref{flag}, general case.) Let $C$ be a flag
  complex.  Given any two points $x,y\in C$ at distance less than
  $\pi$, take a string $\gamma=x_0,\ldots,x_m$ in $C$ from $x$ to $y$
  whose length is less than $\pi$.  Now let $D$ be a finite full
  subcomplex of $C$ that contains the path defined by the string
  $\gamma$.  Then $d_D(x,y)<\pi$ since $D$ contains the string
  $\gamma$.  Hence there is a geodesic string $\gamma'=
  y_0,\ldots,y_n$ from $x$ to $y$ in $D$.  Claim that the string
  $\gamma'$ is also a geodesic in $C$.  To see this suppose not.  In
  that case $C$ would contain a shorter string $\gamma''$.  Now let
  $D'$ be a finite full subcomplex of $C$ containing both $D$ and the
  path defined by $\gamma''$.  By~Theorem~\ref{localisom}, the
  inclusion of $D$ into $D'$ preserves distances less than $\pi$.
  This contradiction shows that $\gamma'$ is a geodesic in $C$, and so
  any two points at distance less than $\pi$ are joined by a geodesic.

  Now suppose one is given a geodesic
  triangle in $C$ of perimeter less than $2\pi$, and two points 
  $x$ and $y$ on the perimeter of this triangle.  Then $x$ and $y$ are
  separated by less than $\pi$ so there exists a geodesic between
  them.  Now the three sides of the triangle together with the
  geodesic from $x$ to $y$ are contained in a finite full subcomplex
  $D$ of $C$.  Since $D$ is flag, the distance from $x$ to $y$ in $D$
  is no greater than the distance between corresponding points on the
  perimeter of a triangle with the same side lengths in the standard
  sphere $S^2$.  But by the argument given above, $d_D(x,y)=d_C(x,y)$ 
  and so the CAT(1) inequality holds.  
\end{proof} 

\begin{proof} (Theorem~\ref{localisom}, general case.)  The proof
  given above in the finite-dimensional case is valid in general now
  that the general case of Theorem~\ref{flag} is known.  
\end{proof} 

\begin{theorem}\label{gromovthm} 
 (Gromov) A simply connected cubical complex is CAT(0) if and
  only if its vertex links are flag complexes.  
\end{theorem} 

\begin{proof} 
The finite-dimensional case is~\cite[II.5.20]{brihae}.  If $C$ is
CAT(0), then by~\cite[II.5.2]{brihae}, for each vertex $v$, $\Lk_C(v)$
is CAT(1).  (Just as in the proof of Theorem~\ref{flag}, the
hypotheses of~\cite[II.5.2]{brihae} do not hold in the
infinite-dimensional case, but Proposition~\ref{epsilon} ensures that
the alternative hypotheses listed in~\cite[II.5.3]{brihae} hold.)  
Hence by Theorem~\ref{flag}, each vertex link is a flag
complex.  

The proof of the general case of `flag links implies CAT(0)' 
will be given below.  
\end{proof}

\begin{theorem} \label{isom} 
Let $C$ be a simply-connected cubical complex with all
  vertex links flag, and let $D$ be a combinatorially convex
  subcomplex of $D$.  Then the map $i:D\rightarrow C$ is isometric.  
\end{theorem} 

\begin{proof} The relationship between the proofs of
  Theorem~\ref{gromovthm} and Theorem~\ref{isom} is similar to the 
relationship between the proofs of Theorem~\ref{flag} and
Theorem~\ref{localisom}.  More precisely, this 
proof will rely on Theorem~\ref{gromovthm}, and so for now it will 
only cover the case when $C$ is finite-dimensional.    
As before, consider the double of 
$C$ along $D$, $C*_DC$.  The argument given in the proof of
Theorem~\ref{localisom} shows that each of the two inclusions of $C$
into $C*_DC$ is an isometric embedding, and there is an isometric 
involution of $C*_DC$ swapping the two copies of $C$ whose fixed point
set is $D$.  By van Kampen's theorem, 
$C*_DC$ is simply-connected.  The link of each vertex of $C*_DC$ is 
flag.  From the cases of Theorem~\ref{gromovthm} that are already proved,
it follows that when $C$ is finite-dimensional, $C*_DC$ is CAT(0).  
Given any two points $x,y\in D$, there is a unique geodesic between 
them in $C*_DC$.  By uniqueness, this geodesic must lie inside $D$, 
and hence $d_D(x,y)=d_C(x,y)$.  
\end{proof} 

\begin{lemma} \label{simpconn} 
Any finite combinatorially convex subcomplex of a combinatorially
CAT(0) cubical complex is contractible.  
\end{lemma} 

\begin{proof} 
Let $D$ be a finite combinatorially convex subcomplex of $C$.  Since
$D$ is finite, the fundamental group of $D$ is finitely generated, 
and so by attaching finitely many discs to $D$ inside $C$ we may kill 
the fundamental group of $D$.  Hence there exists a finite subcomplex 
$D'$ of $C$ such that the inclusion $D\rightarrow D'$ induces the
trivial map on fundamental groups.  By Theorem~\ref{sagthm}, $D'$ is
contained in a finite combinatorially convex subcomplex $E$ of $C$.  
Let $\widetilde E$ be the universal covering of $E$.  Since the map 
$\pi_1(D)\rightarrow \pi_1(E)$ is trivial, there is a subcomplex of 
$\widetilde E$ which is isomorphic to $D$.  Since $\widetilde E$ is 
finite-dimensional, the cases of Theorem~\ref{gromovthm} that are already 
known imply that $\widetilde E$ is
CAT(0).  By the cases of Theorem~\ref{isom} that are already known, 
the map $i:D\rightarrow \widetilde E$ is an isometric embedding.  
Now for any point $x$ in a
CAT(0) space $X$, there is a contraction from $X$ to $\{x\}$ given by 
the homotopy $h_{X,x}:X\times I\rightarrow X$ that sends $(y,t)$ to the point 
at distance $t.d(x,y)$ along the geodesic from $x$ to $y$.  If we pick 
$x$ to be a point of $D$, then the homotopy $h_{\widetilde E,x}$ restricts to 
$D$ as a contraction of $D$ to $x$.  
\end{proof}

\begin{proof} (Theorem~\ref{gromovthm}, general case.)  
Suppose that $C$ is a simply-connected cubical complex in which all
links are flag, and let $x$ and $y$ be points of $C$.  By 
Theorem~\ref{sagthm}, there is a finite combinatorially convex 
subcomplex $D$ of $C$ containing both $x$ and $y$.  The cases of 
Theorem~\ref{gromovthm} that are already proved imply that $D$ is 
CAT(0), since $D$ is simply connected (by Theorem~\ref{isom}) and 
vertex links in $D$ are all flag.  Thus there is a unique geodesic 
in $D$ from $x$ to $y$.  To see that this path is also a geodesic 
in $C$, suppose not.  Then (as in the proof of Theorem~\ref{flag}), 
there is a string $\gamma$ in $C$ of length shorter than $d_D(x,y)$. 
Now let $D'$ be a finite combinatorially convex subcomplex of $C$ that
contains $D$ and the image of $\gamma$.  Then the inclusion of $D$
into $D'$ would not be an isometry, contradicting Theorem~\ref{isom}.  
This contradiction proves that the geodesic from $x$ to $y$ in $D$ 
is a geodesic in $C$, and hence $C$ is a geodesic metric space.  

Now suppose given a geodesic triangle in $C$ and two points $x$ and
$y$ on the perimeter of the triangle.  Let $D$ be a finite
combinatorially convex subcomplex of $C$ containing the triangle and 
the geodesic in $C$ from $x$ to $y$.  Since all links in $D$ are 
flag and by Lemma~\ref{simpconn}, $D$ is simply connected, it follows 
that $D$ is CAT(0).  Hence the distance in $D$, $d_D(x,y)$ is no
greater than the corresponding distance in a Euclidean comparison 
triangle.  But by choice of $D$, $d_D(x,y)=d_C(x,y)$ and so the 
CAT(0) inequality holds in~$C$.  
\end{proof} 

\begin{proof} (Theorem~\ref{isom}, general case.) The proof given
  above in the finite-dimensional case is valid in general now that
  the general case of Theorem~\ref{gromovthm} has been proved.  
\end{proof} 

\begin{theorem}\label{ebarg}
 Let $G$ be a finite group acting by isometries on a 
complete CAT(0) space $X$, or by automorphisms of an arbitrary CAT(0) 
cubical complex $C$.  Then the fixed point set for $G$ is contractible.  
\end{theorem} 

\begin{proof} 
If $x$ and $y$ are fixed points, then the unique geodesic from $x$ to
$y$ consists of fixed points.  Thus if $x$ is a fixed point, the 
geodesic contraction of $X$ or $C$ to the point $x$ restricts to 
a contraction of the fixed point set.  

It remains to prove that the fixed point set is 
non-empty.  In the case when $X$ is complete this is proved 
in~\cite[corollary~II.2.8]{brihae}.  A different proof is given
in~\cite[proposition~3]{bln}. 

If $C$ is a cubical CAT(0) complex, pick a vertex $v$, and construct a
finite combinatorially convex subcomplex $D$ containing the orbit
$G.v$ as in the proof of Theorem~\ref{sagthm}.  By the construction,
$G$ acts on $D$, and by Lemma~\ref{simpconn} and
Theorem~\ref{gromovthm}, $D$ is CAT(0).  Since $D$ is complete, there
is a fixed point in $D$.
\end{proof} 

\section{Cubical and cube complexes}
\label{sec:appthree}

We work throughout with cubical complexes, rather than the more
general complexes referred to as either \emph{cube complexes} or 
\emph{cubed complexes} (see~\cite[example~I.7.40 (4)]{brihae} for 
the definition of a cube complex).  The main purpose of this 
section is to show that every CAT(0) cube complex is in fact 
cubical, and that the second cubical subdivision of any locally 
CAT(0) cube complex is cubical.  

One extra complication that arises in the general case is that vertex
links in cube complexes are not in general simplicial complexes, but
are instead a more general type of complex made from simplices.  By
analogy with the case of cubes, we call these complexes \emph{simplex
complexes}.  A simplex complex is an $M_1$-polyhedral complex in the
sense of~\cite[definition~I.7.37]{brihae} in which each cell is an all
right simplex.  Simplex complexes are rather more general than
$\Delta$-complexes, which we use in Theorem~\ref{thm:deltaversion},
because they lack the ordering on their faces that is implied by a
$\Delta$-complex structure.

\begin{examples} 
There is just one $\Delta$-complex with one 0-cell, one 1-cell and one
2-cell, the so-called `dunces hat', which is contractible.  There is 
another simplex complex with these same numbers of cells, which can be 
obtained by identifying the three sides of a triangle as shown in
Figure~\ref{fig:four}.  

For an example of a locally CAT(0) cube complex that is not a cubical
complex, consider the 2-torus, made by identifying the opposite sides of a
square, as shown in Figure~\ref{fig:four}.  

For an example of a cube complex in which the vertex links are not 
simplicial complexes, consider the 2-sphere made by identifying the 
sides of a square as shown in Figure~\ref{fig:four}.  
\end{examples} 

\begin{figure}
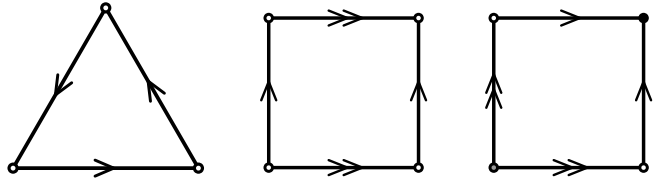

\begin{center}
\includegraphics[scale=1]{exxa.mps}
\qquad\includegraphics[scale=1.4]{exxb.mps}
\qquad\includegraphics[scale=1.4]{exxc.mps}
\end{center}
\caption{\label{fig:four}
From left to right: a simplex complex which is not a $\Delta$-complex; 
the torus as a locally CAT(0) cube complex; the sphere as a cube complex.
}
\end{figure}

\begin{lemma}\label{lem:locgeod}
Let $C$ be a cube complex, let $\gamma:[0,d]\rightarrow I^n$ be a 
straight line path joining two points of $I^n$, and let
$\sigma:I^n\rightarrow C$ be a cube of $C$.  If the endpoints of 
$\sigma\gamma$ are distinct points of $C$, then the image of 
$\sigma\gamma$ is a local geodesic in $C$.  
\end{lemma} 

\begin{proof} 
After possibly replacing $\sigma$ by one of its proper faces, we may
assume that the interior of the line segment $\sigma\gamma$ is
contained in the interior of the cube $\sigma$.  Define $Y_0$, a 
subcomplex of $I^n$, to be the union of all faces of $I^n$ that 
do not meet the image of $\gamma$.  Let $x'=\gamma(0)$ and
$y'=\gamma(d)$ be the endpoints of $\gamma$, and let $x=\sigma(x')$, 
$y=\sigma(y')$ be the endpoints of $\sigma\gamma$.  Define $Y$, a 
closed subset of $I^n$ to be  
\[Y = Y_0 \cup \left(\sigma^{-1}(x) \setminus\{x'\}\right)
\cup \left(\sigma^{-1}(y) \setminus\{y'\}\right).\]  
Choose $\epsilon>0$ so that the distance in $I^n$ between each point
of $Y$ and each point of the image of $\gamma$ is strictly greater 
than $2\epsilon$.  We shall show that each subpath of $\sigma\gamma$ of 
length at most $\epsilon$ is a geodesic in $C$.  

Take $w$, $z$ points of $\sigma\gamma$ separated by an arc of
$\sigma\gamma$ of length $\delta\leq \epsilon$, and suppose that 
$w=w_0,w_1,\ldots,w_m=z$ is a chain from $w$ to $z$.  By the choice of
$\epsilon$, each $w_i$ must either lie in $\sigma$ or in a
higher-dimensional cube of $C$ having $\sigma$ as a face.  If any 
of the $w_i$ lie outside $\sigma$, we may obtain a chain of shorter 
length by replacing each $w_i$ by $z_i$, defined to be the orthogonal
projection of $w_i$ into $\sigma$.  Any chain in $\sigma$ joining 
$w$ and $z$ has length greater than or equal to the arc of
$\sigma\gamma$ from $w$ to $z$, and the claim follows.  
\end{proof} 

\begin{remark} 
Consider the case when $C$ is a M\"obius strip made by identifying 
two opposite sides of a square, $\sigma:I^2\rightarrow C$ is the 
natural identification map, and $\gamma$ is the straight line 
joining two antipodal vertices of the square.  The two endpoints 
of $\gamma$ are mapped to the same point of $C$, and the image of 
$\sigma\gamma$ is not a local geodesic in $C$.  
\end{remark} 

\begin{theorem}\label{thm:cubevcubical}
Let $C$ be any CAT(0) cube complex.  Then $C$ is cubical.  
\end{theorem} 

\begin{proof} 
First we show that each cube of $C$ embeds in $C$ and then that the
intersection of any two cubes of $C$ is a cube of $C$.  Let
$\sigma:I^n\rightarrow C$ be a cube of $C$, and suppose that 
$x',y'$ are points in the boundary of $I^n$ such that
$\sigma(x')=\sigma(y')$.  Let $z'$ be the barycentre of $I^n$, and 
let $\gamma$, $\gamma'$ be the straight line paths in $I^n$ from 
$z'$ to $x'$ and $y'$.  By Lemma~\ref{lem:locgeod}, there exists 
$\epsilon >0$ so that any subpath of either $\sigma\gamma$ or
$\sigma\gamma'$ of length at most $\epsilon$ is a geodesic.  Since 
$C$ is CAT(0), any local geodesic in $C$ is a
geodesic~\cite[II.1.4~(2)]{brihae}, and so $\sigma\gamma$ and 
$\sigma\gamma'$ are geodesics.  Since $\sigma(x')=\sigma(y')$, these 
two geodesics have the same endpoints.  Since geodesics in a CAT(0) 
space are unique, it follows that $x'=y'$ and so $\sigma$ is
injective.  

Next suppose that $\sigma$, $\sigma'$ are distinct cubes of $C$, 
and that $a,b\in C$ are points in the intersection of the images 
of $\sigma$ and $\sigma'$.  Since $\sigma:I^m\rightarrow C$ and 
$\sigma':I^n\rightarrow C$ are injective, there are unique points
$x,y\in \partial I^m$ and $x',y'\in \partial I^n$ such that 
$\sigma(x)=a=\sigma'(x')$ and $\sigma(y)=b=\sigma'(y')$.  Let 
$\gamma$ be the straight line in $I^m$ from $x$ to $y$, and let 
$\gamma'$ be the straight line in $I^n$ from $x'$ to $y'$.  By 
Lemma~\ref{lem:locgeod}, each of $\sigma\gamma$ and $\sigma'\gamma'$
is a local geodesic in $C$, and so by~\cite[II.1.4~(2)]{brihae} 
each is a geodesic in $C$.  These 
geodesics have the same endpoints, and so they must coincide.  
It follows that the intersection of the images of $\sigma$ and 
$\sigma'$ is a convex subset of $C$, and so must be a single 
face of $\sigma$ and of~$\sigma'$.  
\end{proof} 

\begin{lemma} \label{lem:cube}
  Let $C$ be a cube viewed as a cubical complex, and let $\mu$ be a
  cube of the cubical subdivision $C'$ of $C$.  The intersection of
  those faces of $C$ that meet $\mu$ is non-empty.
\end{lemma} 

\begin{proof} 
There is a natural bijective correspondence between vertices of
$C'$ and faces of $C$.  Containment of faces of $C$
induces a partial ordering on vertices of $C'$ such that 
the vertex set of any cube $\mu$ of $C'$ has unique maximal and minimal 
elements.  This establishes a bijection between cubes of 
$C'$ and ordered pairs $(\tau_1,\tau_2)$ of faces of $C$ 
with $\tau_1\geq \tau_2$: the pair $(\tau_1,\tau_2)$ corresponds to 
the cube of dimension $\dim(\tau_1)-\dim(\tau_2)$ whose maximal and
minimal vertices are the barycentres of $\tau_1$ and $\tau_2$.  
With this notation, $\mu=(\tau_1,\tau_2)$ meets $\tau$ if and only 
if $\tau\geq \tau_2$, and in this case the intersection $\mu\cap
\tau$ is equal to the cube $(\tau_1\cap\tau,\tau_2)$ of $C'$.  
Hence every face of $C$ that meets $\mu=(\tau_1,\tau_2)$ contains
$\tau_2$, and the intersection of all such faces is equal to
$\tau_2$.  
\end{proof}

\begin{proposition} \label{prop:subdiv}
Let $C$ be a cube complex with the property that 
every vertex link in $C$ is a simplicial complex, and let $C'$, $C''$
 denote the first and second cubical subdivisions of $C$.  Then 
\begin{enumerate} 

\item Cubes of $C'$ embed

\item The intersection of any two cubes of $C'$ is a disjoint union of
  cubes of~$C'$

\item $C''$ is a cubical complex 

\end{enumerate} 
\end{proposition} 

\begin{remark} 
The link of a vertex of a cube complex is unchanged by passing to the 
cubical subdivision.  So if $C$ has a vertex link that is not
simplicial, then so does every cubical subdivision of $C$.  
\end{remark}  

\begin{proof} 
Any cube $\mu$ of $C'$ is a face of another cube $\nu$ of $C'$ such
that $\nu$ contains some vertex $v$ of $C$ amongst its vertices.  Thus 
it suffices to show that such a $\nu$ embeds.  It suffices to show
that the $2^n$ vertices of $\nu$ are distinct.  The line segments in 
$\nu$ from $v$ to the other $2^n-1$ vertices of $\nu$ define $2^n-1$ 
points of $\Lk_{C'}(v)=\Lk_C(v)$, and it remains to show that these 
points are distinct.  If $\tau$ is the $(n-1)$-simplex of
$\Lk(v)$ defined by $\nu$, then these points are the barycentres of 
$\tau$ and its faces.  Since $\Lk(v)$ is simplicial, these points are 
distinct.  Hence the vertices of $\nu$ lie in the interiors of
distinct cubes of $C$, and so they are distinct.  

Let $\sigma$ and $\sigma'$ be cubes of $C'$, and let $v$ be a vertex of 
the intersection $\sigma\cap\sigma'$.  Define a set $S$ of vertices of 
$C'$ by 
\[S=\{w\in (C')^0 : \hbox{$(v,w)$ is an edge of
  $\sigma\cap\sigma'$}\}.  
\] 
Since $v$ was an arbitrary vertex of $\sigma\cap\sigma'$, it suffices
to show that there is a cube contained in $\sigma\cap\sigma'$ whose 
vertex set contains $S\cup\{v\}$.  By hypothesis, there is a cube 
$\mu\leq \sigma$ such that $S\cup\{v\}$ is contained in the vertex set 
of $\mu$, and we may take $\mu$ of minimal dimension amongst all such 
cubes, in which case $\mu$ is an $|S|$-cube.  
Similarly, there is an $|S|$-cube $\mu'\leq \sigma'$ with 
$S\cup \{v\}$ contained in its vertex set.  The images of $\mu$ and
$\mu'$ in $\Lk_{C'}(v)$ are simplices with the same vertex set.  Since 
$\Lk_{C'}(v)$ is simplicial, it follows that $\mu=\mu'$, and so $\mu$ 
is contained in $\sigma\cap\sigma'$ as required.  

It remains to show that $C''$ is cubical.  Let $\mu$ be a cube of 
$C''$, and let $\nu$ be a cube of $C'$ that contains $\mu$.  Since 
$\nu$ embeds, it follows that $\mu$ embeds.  Now let $\mu'$ be another
cube of $C''$, and let $\nu'$ be a cube of $C'$ that contains $\mu'$.  
The intersection $\mu\cap\mu'$ is contained in $\nu\cap\nu'$.  By 
Part~2, the intersection $\nu\cap\nu'$ is equal to a disjoint union 
$\sigma_1\coprod \cdots\coprod \sigma_r$ of cubes of $C'$.  By
Lemma~\ref{lem:cube}, the set of faces of $\nu$ that meet $\mu$ is a 
connected subcomplex of $\nu$.  It follows that the $\mu$ intersects 
at most one of the cubes $\sigma_i$ non-trivially.  Similarly, $\mu'$ 
intersects $\sigma_j$ for at most one value of $j$.  Hence either 
$\mu\cap\mu'$ is empty, or there exists $i$ so that $\mu\cap\mu' 
\subseteq \sigma_i$.  In this latter case, since $\sigma_i$ embeds 
in $C'$, it follows that $\mu\cap\mu'$ is a face of both $\mu$ and
$\mu'$.  
\end{proof}

For a vertex $v$ in a cube complex $C$, define $\nstar_C(v)$ to the the
union of all of the closed cubes of $C$ that contain $v$.  For $m$ a 
positive integer let $C^{(m)}$ denote the $m$th iterated cubical
subdivision of $C$.  

\begin{proposition}\label{starprop}  
For any cube complex $C$ and vertex $v$, $\nstar_{C''}(v)$ is contractible.
For any $m\geq 2$ there is an isomorphism
$\phi_m:\nstar_{C''}(v)\rightarrow \nstar_{C^{(m)}}(v)$ which multiplies
  distances by $2^{2-m}$.  
\end{proposition} 

\begin{proof} 
Let $\sigma$ be an $n$-cube of side length one, and let $V$ be the 
vertex set of $\sigma$.  For each $v\in V$,  and each $m\geq 0$, 
$\nstar_{\sigma^{(m)}}(v)$ is an $n$-subcube of $\sigma$ of side
length $2^{-m}$ containing $v$ as a vertex.  Moreover if $m\geq 2$, 
then the union 
$$X(\sigma,m)=\bigcup_{v\in V}\nstar_{\sigma^{(m)}}(v)$$ 
is a disjoint 
union of $2^n$ $n$-cubes, each containing a unique vertex from $V$.
It follows that there is a linear homotopy $H^\sigma$ from the 
identity map of $X(\sigma,2)$ to the locally constant map which 
sends each component of $X(\sigma,2)$ to the element of $V$ that 
it contains.  If $\tau$ is a face of $\sigma$, then $X(\tau,2)$
defined similarly is naturally a subspace of $X(\sigma,2)$ and 
$H^\sigma$ and $H^\tau$ agree on $X(\tau,2)$.  

Now if $C$ is any cube complex, define 
$$X(C,m)= \bigcup_{\sigma} X(\sigma,m),$$ 
where $\sigma$ ranges over all of the cubes of $C$.  
By the above remark, the homotopies $H^\sigma$ fit together to 
make a single homotopy $H^C$ from the identity map of $X(C,2)$ 
to the constant map that sends each component of $X(C,2)$ to the 
unique vertex of $C$ that it contains.  Since $H^C$ was defined 
linearly, if we put $t=1-4/2^m$, then for any $m\geq 2$, the 
map $H^C_t$ is an isomorphism from $X(C,2)$ to $X(C,m)$, and 
for any vertex $v$ of $C$ we may define $\phi_m$ to be the 
restriction of $H^C_t$ to the component $\nstar_{C''}(v)$ of 
$X(C,2)$.  
\end{proof}  

\begin{theorem} \label{thm:gromgen}
Let $C$ be a cube complex, with arbitrary vertex links.  Then 
$C$ is locally CAT(0) if and only if all vertex links in $C$ 
are flag simplicial complexes.  
\end{theorem} 

\begin{proof} 
If $C$ is a cube complex in which all vertex links are flag 
simplicial complexes, then $C''$ is a cubical complex with 
this property, and the universal cover $\widetilde C''$ of 
$C''$ satisfies the hypotheses of Theorem~\ref{gromovthm}.  
The universal cover $\widetilde C$ of $C$ is isometric to 
$\widetilde C''$, and so it too is CAT(0).  It follows that 
$C$ is locally CAT(0).  

Conversely, suppose that $C$ is a cube complex with a vertex 
$v$ such that the link of $v$ is not a flag simplicial complex.  
It suffices to show that for all $\epsilon$, the metric ball 
$B(v,\epsilon)$ around $v$ is not CAT(0).  Let $S$ denote 
$\nstar_{C''}(v)$, and note that $B(v,1/4)$ is contained in $S$.  
By Proposition~\ref{starprop} $S$ is contractible, and either 
$S$ is cubical but has a non-flag link or $S$ is not cubical.  
It follows from either Theorem~\ref{thm:cubevcubical} or
Theorem~\ref{gromovthm} that $S$ cannot be CAT(0).  Hence there 
exists a geodesic triangle $T$ in $S$ that violates the CAT(0) 
inequality.  Now pick any $\epsilon\leq 1/12$, and let $m$ be 
sufficiently large that $\phi_m(T)$ is contained in $B(v,\epsilon)$
where $\phi_m$ is the map in Proposition~\ref{starprop}.  Since 
$\phi_m(T)$ violates the CAT(0) inequality in $\nstar_{C^{(m)}}(v)$, 
it suffices to show that the embedding of $\phi_m(T)$ into $C$ is 
isometric.  Since $\phi_m(T)$ is contained in $B(v,\epsilon)$, 
it follows that the distance in $\nstar_{C^{(m)}}$ between any two 
points of $\phi_m(T)$ is at most $2\epsilon\leq 1/6$.  If there 
were a shorter path in $C$ between two points of $\phi_m(T)$, 
it could not leave $B(v,1/4)$.  Since $B(v,1/4)\subset \nstar_{C''}(v)$ 
and the inclusion of $\nstar_{C^{(m)}}(v)$ into $\nstar_{C''}(v)$ is 
an isometry, it follows that $\phi_m(T)$ embeds isometrically into 
$C$, and so $B(v,\epsilon)$ is not CAT(0).  
\end{proof}  

\begin{remark} 
The special case of Theorem~\ref{thm:gromgen} in which $C$ is 
assumed to be finite-dimensional is given in Gromov's original 
paper~\cite[section~4.2.C]{gromov} and 
in~\cite[theorem~II.5.20]{brihae}.  
\end{remark}

\begin{corollary}\label{cor:secsubdiv}
The second cubical subdivision of any locally CAT(0) cube complex is
cubical.  
\end{corollary} 

\begin{proof} 
If $C$ is a locally CAT(0) cube complex, then
Theorem~\ref{thm:gromgen} shows that $C$ satisfies the hypotheses of 
Proposition~\ref{prop:subdiv}, and hence $C''$ is cubical.  
\end{proof}

\bibliography{mktjt}
\bibliographystyle{abbrv} 

\affiliationone{Ian J. Leary\\
Department of Mathematics\\
The Ohio State University\\
231 West 18th Avenue\\
Columbus\\
Ohio 43210\\
USA}
\affiliationtwo{~}
\affiliationthree{Current address:\\
School of Mathematics\\
University of Southampton\\
Southampton\\
SO17 1BJ\\
UK
\email{i.j.leary@soton.ac.uk}}

\end{document}